\documentclass[a4paper,11pt]{amsart}
\usepackage{amsmath,amsfonts,amsthm,amssymb,enumerate,ifpdf}
\usepackage{graphicx}

\numberwithin{equation}{section}
\newtheorem{theorem}{Theorem}[section]
\newtheorem{corollary}[theorem]{Corollary}
\newtheorem{proposition}[theorem]{Proposition}
\newtheorem{conjecture}[theorem]{Conjecture}

\newtheorem{lemma}[theorem]{Lemma}

{
\theoremstyle{definition}
\newtheorem{definition}[theorem]{Definition}
\newtheorem{example}[theorem]{Example}
\newtheorem{remark}[theorem]{Remark}
\newtheorem{assumption}[]{Assumption}

}

\title{Simplifying branched covering surface-knots by an addition of 1-handles with chart loops}
\author{Inasa Nakamura}
\address{Graduate School of Mathematical Sciences, The University of Tokyo\newline
3-8-1 Komaba, Tokyo 153-8914, Japan}
\email{inasa@ms.u-tokyo.ac.jp}

\begin{document}

\begin{abstract}
A branched covering surface-knot over an oriented surface-knot $F$ is a surface-knot in the form of a branched covering over $F$. 
A branched covering surface-knot over $F$ is presented by a graph called a chart on a surface diagram of $F$. 
For a branched covering surface-knot, an addition of 1-handles equipped with chart loops is a simplifying operation which deforms the chart to the form of the union of free edges and 1-handles with chart loops. We investigate properties of such simplifications. 
\end{abstract}

\maketitle

\section{Introduction}\label{sec1}

A {\it surface-knot} is the image of a smooth embedding of a closed surface into the Euclidean 4-space $\mathbb{R}^4$ \cite{CKS, Kamada02}. 
In this paper, we consider oriented surface-knots. 
For a surface-knot $F$, we consider a surface in the form of a covering over $F$, called a branched covering surface-knot or a 2-dimensional braid over $F$ \cite{N, N4}. The first aim of this paper is to announce that we change the name of a \lq\lq 2-dimensional braid" to a \lq\lq branched covering surface-knot"  (see Section \ref{sec2}). Two branched covering surface-knots over $F$ are equivalent if one is carried to the other by an ambient isotopy of $\mathbb{R}^4$ whose restriction to a tubular neighborhood of $F$ is fiber-preserving. 
A branched covering surface-knot over $F$, denoted by $(F, \Gamma)$, is presented by a graph $\Gamma$ called a chart on a surface diagram of $F$ \cite{N4}. 

In \cite{N5}, we showed that a branched covering surface-knot deforms to a simplified form in terms of charts by an addition of 1-handles with chart loops.  
The second aim of this paper is to investigate further such simplifications of branched covering surface-knots.

Let $B^2$ be a unit 2-disk and let $I=[0,1]$. 
A {\it 1-handle} is a 3-ball $h=B^2 \times I$ smoothly embedded in $\mathbb{R}^4$ such that $h \cap F=B^2 \times \partial I$. In particular, a 1-handle $h$ is said to be {\it trivial} if it is  embedded in a 3-ball $B^3$ such that $F \cap B^3$ is a 2-disk. 
The surface-knot obtained from $F$ by a {\it 1-handle addition} along $h$ is the surface
\[
(F-(\mathrm{Int} B^2 \times \partial I)) \cup (\partial B^2 \times I),
\]
which is denoted by $F+h$. In this paper, we assume that $h$ is orientable, that is, $F+h$ is orientable, and we give $F+h$ the orientation induced from that of $F$. Further, we assume that 1-handles are trivial. 
Let $h=B^2 \times I$ be a 1-handle. We call $B^2 \times \{0\}$ and $B^2 \times \{1\}$ the {\it ends} of $h$. 
When both ends of $h$ are on a 2-disk $E$ in $F$, we determine the {\it core loop} of $h$ as follows. Let $\rho(t), t \in I$ be an oriented path in $\partial B^2 \times I \subset h$ with the orientation of $I$ such that $\rho(t) \in \partial B^2 \times \{t\}, t \in I$. 
We define the {\it core loop} of $h$ by the oriented closed path in $F+h$ obtained from $\rho(t), t \in [0,1]$ by connecting the initial and terminal points $\rho(0)$ and $\rho(1)$ by a simple arc in $E$, with the induced orientation. 
We determine the {\it cocore} of a 1-handle $h$ by the oriented closed path $\partial B^2 \times \{0\} \subset h$, with the orientation of $\partial B^2$. 
Further, we determine the base point of the core loop and the cocore of $h$ by their intersection point. See Fig. \ref{fig1}. For a set of 1-handles, we add a condition that core loops are mutually disjoint. 

\begin{figure}[ht]
\centering
\includegraphics*[height=3.5cm]{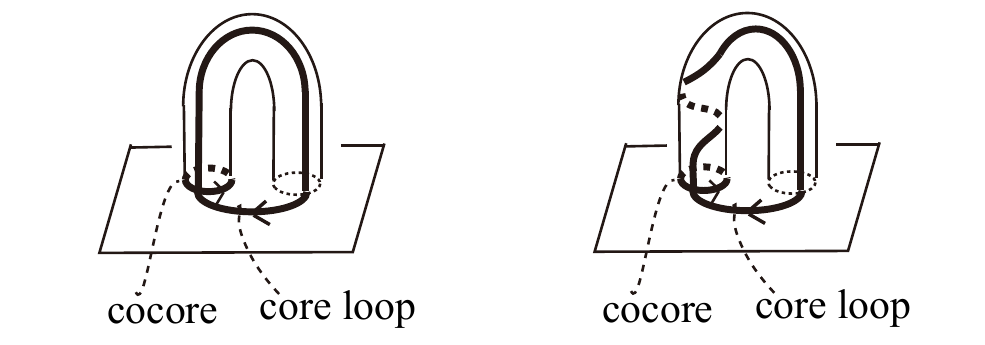}
\caption{The core loop and the cocore of a 1-handle. There are two types.}
\label{fig1}
\end{figure}

\begin{assumption}\label{assump1}
In this paper, for simplicity, we assume that all 1-handles are trivial, and we assume that a surface-knot $F$ is an unknotted surface-knot with one component in the standard form, that is, $F$ is in the form of the boundary of a handlebody in $\mathbb{R}^3 \times \{0\} \subset \mathbb{R}^4$. 
\end{assumption}

A 1-handle admits two types framings, presented by the core loop and the cocore as indicated in Fig. \ref{fig1} \cite{Boyle2, Livingston}, see also \cite[Lemma 4.2]{N5}. 
In this paper, for simplicity, we do not distinguish the framings. 
We consider a 1-handle $h$ attached to a 2-disk $E$. 
For commutative braids $a$ and $b$, we denote by $h(a,b)$ a covering surface over $h$ whose cocore with the given orientation and core loop with the reversed orientation present the braids $a$ and $b$, respectively. Such a chart is drawn on the union of $h$ and neighborhoods of the core loop and the cocore in $E$,   
and this chart is unique up to C-move equivalence \cite{N}, where we assume that $E$ has no edges nor vertices of a chart except those coming from $h(a,b)$. Hence, the notation $h(a,b)$ is well-defined, under the condition that we assign the place where we attach the 1-handle. We call the chart on $E+h$ a {\it chart of a 1-handle}, and we call $h(a,b)$ a {\it 1-handle with a chart}, or simply a {\it 1-handle}. 
A {\it chart loop} is a closed path consisting of an edge of a chart or diagonal edges of a chart connected with vertices of degree 4 (crossings). 
In particular, when a chart of a 1-handle consists of chart loops, we call such a 1-handle a {\it 1-handle with chart loops}. 

\begin{remark}
We repeat that we do not distinguish the framings of 1-handles. 
If we take the core loop and the cocore of a 1-handle as indicated in the left figure of Fig. \ref{fig1}, then a 1-handle $h(a,b)$ with the core loop and the cocore as in the right figure of Fig. \ref{fig1} is denoted by $h(a, ab)$. For simplicity, we denote both types of 1-handles by $h(a,b)$.
\end{remark}

An edge of a chart is called a {\it free edge} if its end points are vertices of degree one (black vertices). 
Let $\Gamma_0$ be a chart consisting of a disjoint union of several free edges. 
For a branched covering surface-knot $(F, \Gamma_0)$ and 1-handles with charts $h(a_1, b_1)$, $\ldots$, $h(a_g, b_g)$ attached to a 2-disk $E$ in $F$, where there are no edges nor vertices of charts on $E$ except those of the attached 1-handles, we denote the branched covering surface-knot which is the result of the 1-handle addition by $(F', \Gamma')=(F, \Gamma_0) + \sum_{i=1}^g h(a_i,b_i)$. Note that since $\Gamma_0$ is a disjoint union of free edges, the presentation is well-defined.  

We showed in \cite{N5} the following results. 
Under Assumption \ref{assump1}, the results are written as follow. Let $N$ be a positive integer.

\begin{theorem}[{\cite[Theorem 1.6]{N5}}]\label{thm1-2}
Let $(F, \Gamma)$ be a branched covering surface-knot of degree $N$. 
By an addition of finitely many 1-handles in the form $h(\sigma_i, e)$ or $h(e,e)$ $(i \in \{1, \ldots,N-1\})$, to appropriate places in $F$, $(F, \Gamma)$ deforms to 
\begin{equation}\label{eq1-1}
(F, \Gamma_0) +\sum_{k} h(\sigma_{i_k}, e)+\sum_l h(\sigma_{i_l}, \sigma_{j_l}^{\epsilon_l})+\sum h(e,e),
\end{equation}
where $i_k, i_l, j_l \in \{1,\ldots,N-1\}, |i_l-j_l|>1$ and $\epsilon_l\in \{+1, -1\}$, and $\Gamma_0$ is a chart consisting of several (maybe no) free edges. 

In particular, by an addition of 1-handles $\sum_{i=1}^{N-1}h(\sigma_i, e)$ and finitely many $h(e,e)$, to a fixed 2-disk in $F$, $(F, \Gamma)$ deforms to 
\begin{equation}\label{eq1-2}
(F, \Gamma_0) +\sum_{i=1}^{N-1}h(\sigma_i, e)+\sum_l h(\sigma_{i_l}, \sigma_{j_l}^{\epsilon_l})+\sum h(e,e),
\end{equation}
where $i_l, j_l \in \{1,\ldots,N-1\}, |i_l-j_l|>1$ and $\epsilon_l \in \{+1, -1\}$, and $\Gamma_0$ is a chart consisting of several free edges. 
\end{theorem}

\begin{definition}\label{def1-7}
We call $(F, \Gamma)$ in the form (\ref{eq1-1}) a branched covering surface-knot in a {\it weak simplified form}, and  
we call the minimal number of 1-handles necessary to deform $(F, \Gamma)$ to the form (\ref{eq1-1}) the {\it weak simplifying number} of $(F, \Gamma)$, which is denoted by $u_w(F, \Gamma)$. 
\end{definition}

\begin{theorem}[{\cite[Theorem 1.8]{N5}}]\label{thm1-4}
Let $(F, \Gamma)$ be a branched covering surface-knot of degree $N$.  By an addition of finitely many 1-handles in the form $h(\sigma_i, e)$, $h(\sigma_i, \sigma_j^\epsilon)$ or $h(e,e)$ $(i, j \in \{1, \ldots,N-1\}, |i-j|>1, \epsilon \in \{+1, -1\})$, to appropriate places in $F$, $(F, \Gamma)$  deforms to  
\begin{equation}\label{eq1-3}
(F, \Gamma_0) +\sum_k h(\sigma_{i_k}, e)+\sum h(e,e),
\end{equation}
where $i_k \in \{1,\ldots,N-1\}$ and $\Gamma_0$ is a chart consisting of several free edges. 

In particular, by an addition of 1-handles $\sum_{i=1}^{N-1}h(\sigma_i, e)$ and finitely many 1-handles in the form $h(\sigma_i, \sigma_j^\epsilon)$ $(i, j \in \{1, \ldots,N-1\}, |i-j|>1, \epsilon \in \{+1, -1\})$ or $h(e,e)$, to a fixed 2-disk in $F$, $(F, \Gamma)$ deforms to 
\begin{equation}
(F, \Gamma_0) +\sum_{i=1}^{N-1}h(\sigma_i, e)+\sum h(e,e),   
\end{equation}
where $\Gamma_0$ is a chart consisting of several free edges. \end{theorem}

\begin{definition}\label{def1-9}
We call $(F, \Gamma)$ in the form (\ref{eq1-3}) a branched covering surface-knot in a {\it  simplified form}, and 
we call the minimal number of 1-handles necessary to deform $(F, \Gamma)$ to the form (\ref{eq1-3}) the {\it simplifying number} of $(F, \Gamma)$, which is denoted by $u(F, \Gamma)$.  
\end{definition}

\begin{remark}
Here, we give new names for $u_w(F, \Gamma)$ and $u(F, \Gamma)$. 
In \cite{N5}, we called $u_w(F, \Gamma)$ and $u(F, \Gamma)$ the weak unbraiding number and the unbraiding number, respectively. 
\end{remark}
  
\begin{figure}[ht]
\centering
\includegraphics*[width=13cm]{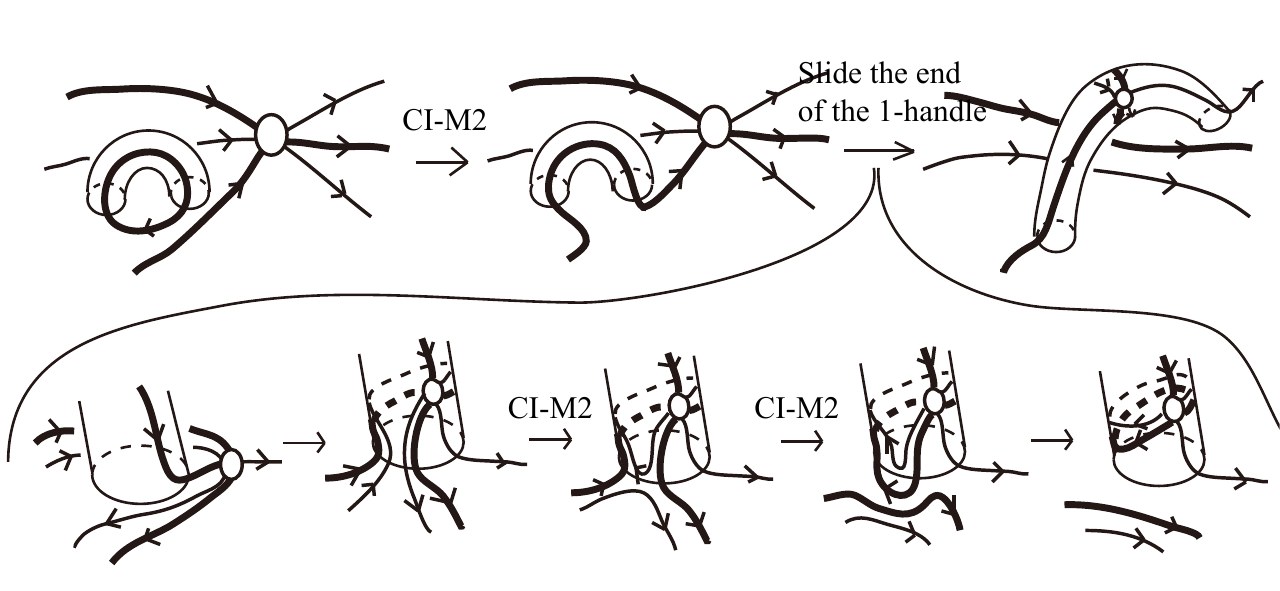}
\caption{Collecting a white vertex on a 1-handle. For simplicity, we omit labels of chart edges.}
\label{fig2a}
\end{figure}

Our results are as follow. Let $(F, \Gamma)$ be a branched covering surface-knot. Let $N$ be the degree of the chart $\Gamma$. We denote by $b(\Gamma)$, $c(\Gamma)$, and $w(\Gamma)$ the numbers of black vertices, crossings, and white vertices, respectively. 

In \cite[Proposition 1.11]{N5}, we obtained the inequality $u_w(F, \Gamma) \leq w(\Gamma)+2c(\Gamma)+N-1$. We showed it by deforming $(F, \Gamma)$ to a weak simplified form 
by sliding an end of a 1-handle to collect white vertices as indicated in Fig. \ref{fig2a}, and eliminating white vertices and then chart loops by a certain method; see \cite{N5}. By deforming $(F, \Gamma)$ to a weak simplified form by a new method, we obtain the following inequality.  

\begin{theorem}\label{thm1-7}
Let $(F, \Gamma)$ be a branched covering surface-knot of degree $N$. Let $b(\Gamma)$, $c(\Gamma)$, and $w(\Gamma)$ be the numbers of black vertices, crossings, and white vertices of $\Gamma$, respectively. 
We have 
\begin{equation}\label{eq1-4}
u_w(F, \Gamma)\leq \max \left\{ \left\lfloor \frac{w(\Gamma)}{2}+\frac{b(\Gamma)}{4} (N-2) \right\rfloor, c(\Gamma) \right\} +N-1, 
\end{equation}
where $\lfloor x\rfloor$ is the largest integer less than or equal to $x$. 
In particular, if $b(\Gamma)=0$, then 
\begin{equation}\label{eq1-5}
u_w(F, \Gamma)\leq \max \left\{ \frac{w(\Gamma)}{2}, c(\Gamma)\right\} +N-1.
\end{equation}
\end{theorem}

\begin{theorem}\label{thm1-8}
Let $(F, \Gamma)$ be a branched covering surface-knot of degree $N$. 
 By an addition of finitely many 1-handles in the form $h(\sigma_i, e)$, or $h(e,e)$ $(i \in \{1, \ldots,N-1\})$, to appropriate places in $F$, $(F, \Gamma)$  deforms to  
\begin{equation}\label{eq:7}
(F, \Gamma_0) +\sum_{i=1}^{N-1}h(\sigma_i, e)+ \sum h(\sigma_1, \sigma_3^{\epsilon}) +\sum h(e,e), 
\end{equation}
where $\Gamma_0$ is a chart consisting of several free edges, and $\epsilon$ is a fixed sign valued in $\{ +1, -1\}$. 
\end{theorem}

\begin{theorem}\label{thm1-9}
Let $(F, \Gamma)$ be a branched covering surface-knot of degree $N$. 
 By an addition of finitely many 1-handles in the form $h(\sigma_i, e)$, or $h(e,e)$ $(i \in \{1, \ldots,N-1\})$, to appropriate places in $F$, $(F, \Gamma)$  deforms to one of the followings: 
\begin{equation}\label{eq1-8x}
(F, \Gamma_0) +\sum_{i=1}^{N-1}h(\sigma_i, e)+\sum h(e,e),  
\end{equation}
or 
\begin{equation}\label{eq1-9x}
(F, \Gamma_0) +h(\sigma_1, \sigma_3)+\sum_{i=2}^{N-1}h(\sigma_i, e)+\sum h(e,e), 
\end{equation}
where $\Gamma_0$ is a chart consisting of several free edges. 

In particular, if $w(\Gamma)=0$ or $b(\Gamma)=0$, then $(F, \Gamma)$ deforms to $(\ref{eq1-8x})$ if $c(\Gamma)$ is even, and $(\ref{eq1-9x})$ if $c(\Gamma)$ is odd. Here $w(\Gamma)$, $b(\Gamma)$ and $c(\Gamma)$ are the numbers of white vertices, black vertices and crossings of $\Gamma$, respectively. We write down the case when $b(\Gamma)=0$ as follows. If $b(\Gamma)=0$, then $(F, \Gamma)$ deforms to   
\begin{equation}\label{eq1-8y}
(F, \emptyset) +\sum_{i=1}^{N-1}h(\sigma_i, e)+\sum h(e,e), \text{ if $c(\Gamma)$ is even,} 
\end{equation}
and  
\begin{equation}\label{eq1-9y}
(F, \emptyset) +h(\sigma_1, \sigma_3)+\sum_{i=2}^{N-1}h(\sigma_i, e)+\sum h(e,e), \text{ if $c(\Gamma)$ is odd.} 
\end{equation}
\end{theorem}

\begin{corollary}\label{thm1-11}
Let $(F, \Gamma)$ be a branched covering surface-knot of degree $N$. Let $w(\Gamma)$ and  $b(\Gamma)$ be the numbers of white vertices and black vertices of $\Gamma$, respectively. 

If $w(\Gamma)=0$, then 

\begin{equation}\label{eq1-12}
u_w(F, \Gamma)\leq N. 
\end{equation}

Further, 
\begin{equation}\label{eq1-10}
u_w(F, \Gamma)\leq \max\left\{ 1,\left \lfloor \frac{w(\Gamma)}{2}+\frac{b(\Gamma)}{4} (N-2)\right\rfloor \right\}+N-1, 
\end{equation}
where $\lfloor x\rfloor$ is the largest integer less than or equal to $x$. 
In particular, if $b(\Gamma)=0$, then
\begin{equation}\label{eq1-11}
u_w(F, \Gamma)\leq \max \left\{ 1, \frac{w(\Gamma)}{2} \right\} +N-1.
\end{equation}

 Further, the inequalities $(\ref{eq1-12})$, $(\ref{eq1-10})$ and $(\ref{eq1-11})$ also hold true for $u(F, \Gamma)$. 
\end{corollary}

\begin{corollary}\label{thm1-10}
Let $(F, \Gamma)$ be a branched covering surface-knot of degree $N$. 
Then 
\begin{equation}\label{eq1-8a}
u(F, \Gamma)\leq 
\max\{1, u_w(F, \Gamma)\}+N-1. 
\end{equation}
\end{corollary}

In \cite[Proposition 1.10]{N5}, we announced that $u_w(F, \Gamma) \leq u(F, \Gamma) \leq u_w(F, \Gamma)+c_{\mathrm{alg}}(\Gamma)$, where $c_{\mathrm{alg}}(\Gamma)$ is the sum of the absolute values of the numbers each of which is the number of crossings of type $c_{i,j}$ minus that of type $c_{j,i}$ ($i<j$, $i,j \in \{1, \ldots, N-1\}$); see Remark \ref{rem-4}. In the proof for $u_w(F, \Gamma) \leq u(F, \Gamma)$, we said it is obvious, but it is not. 
We thought that by an addition of 1-handles in the form $h(\sigma_i, e)$ or $h(e,e)$, we first deform $(F, \Gamma)$ to a branched covering surface-knot $(F', \Gamma')$ in a weak simplified form (\ref{eq1-1}), and then by an addition of 1-handles in the form $h(\sigma_i, \sigma_j^\epsilon)$ ($|i-j|>1, \epsilon \in \{+1, -1\}$), we deform $(F', \Gamma')$ to a simplified form (\ref{eq1-3}). We forgot that we can add $h(\sigma_i, \sigma_j^\epsilon)$ from the first. We correct the statement to a conjecture. 
\begin{conjecture}\label{conj}
Let $(F, \Gamma)$ be a branched covering surface-knot. We conjecture that 
\begin{equation}
u_w(F, \Gamma) \leq u(F, \Gamma).
\end{equation}
\end{conjecture}

For $u(F, \Gamma) \leq u_w(F, \Gamma)+c_{\mathrm{alg}}(\Gamma)$, the proof given in \cite{N5} was for a special case; see Remark \ref{rem-4}. 
We correct the statement to a conjecture.

\begin{conjecture}\label{conj2}
Let $(F, \Gamma)$ be a branched covering surface-knot of degree $N$. We conjecture that 
\begin{equation}\label{conj5}
u(F, \Gamma)\leq u_w(F, \Gamma)+c_{\mathrm{alg}}(\Gamma),
\end{equation}
where $c_{\mathrm{alg}}(\Gamma)$ is the sum of the absolute values of the numbers each of which is the number of crossings of type $c_{i,j}$ minus that of type $c_{j,i}$ 
$(i<j, i,j \in \{1, \ldots, N-1\})$; see Remark \ref{rem-4}. 
Further, by the proof of Corollary \ref{thm1-10}, we conjecture that 
\begin{equation}
u(F, \Gamma)\leq \max \{u_w(F, \Gamma), N-1\}+1,  
\end{equation}
and further, we conjecture that 
\begin{equation}
u(F, \Gamma)\leq \max \{u_w(F, \Gamma), N-1\}.  
\end{equation}
\end{conjecture}

Here, we give a simple example satisfying $u_w(F, \Gamma) < u(F, \Gamma)$. 
\begin{example}
Let $N>1$ be an integer. 
Let $(F, \Gamma)$ be a branched covering surface-knot $(S^2, \emptyset)+\sum_{i=1}^N h(\sigma_i, \sigma_{N+i})$ of degree $2N+1$. 
Then, $u_w(F, \Gamma)=0$ and $u(F, \Gamma)=N$. 
\end{example}

Concerning Conjectures \ref{conj}, \ref{conj2} and the simplifying number and the weak simplifying number, we discuss in \cite{N7}. We also investigate concerning Conjecture \ref{conj2} and the simplifying number and the weak simplifying number for branched covering surface-knots with black vertices in \cite{N8}. 
 
The paper is organized as follows. In Section \ref{sec2}, we review branched covering surface-knots and their chart presentations. 
In Section \ref{sec3}, we show Theorem \ref{thm1-7}. In Section \ref{sec4}, we show Theorem \ref{thm1-8}. In Section \ref{sec5}, we show Theorem \ref{thm1-9} and Corollaries \ref{thm1-11}, \ref{thm1-10}. 

\section{Terminology change: branched covering surface-knots (formerly 2-dimensional braids) and their chart presentations}\label{sec2}
In this section, we review a branched covering surface-knot, formerly called a 2-dimensional braid over a surface-knot \cite{N4}, which is an extended notion of 2-dimensional braids or surface braids over a 2-disk \cite{Kamada92, Kamada02, Rudolph}. A branched covering surface-knot over a surface-knot $F$ is presented by a finite graph called a chart on a surface diagram of $F$ \cite{N4} (see also \cite{Kamada92, Kamada02}). For two branched covering surface-knots of the same degree, they are equivalent if their surface diagrams with charts are related by a finite sequence of ambient isotopies of $\mathbb{R}^3$, and local moves called C-moves \cite{Kamada92, Kamada96, Kamada02} and Roseman moves \cite{N4} (see also \cite{Roseman}). 

\begin{remark}
We formerly called a branched covering surface-knot a 2-dimensional braid over a surface-knot. The term \lq\lq 2-dimensional braid" is usually used for a surface properly embedded in a bidisk in the braid form over a 2-disk \cite{Kamada02}, and it has a boundary, but a branched covering surface-knot is embedded in $\mathbb{R}^4$ and has no boundary. Hence, in order to avoid confusion, we change the notation. 
\end{remark}

We work in the smooth category. A {\it surface-knot} is the image of a smooth embedding of a closed surface into $\mathbb{R}^4$. Two surface-knots are {\it equivalent} if one is carried to the other by an ambient isotopy of $\mathbb{R}^4$. 
 
\subsection{Branched covering surface-knots over a surface-knot}
 
Let $B^2$ be a 2-disk, and let $N$ be a positive integer. 
For a surface-knot $F$, let $N(F)=B^2 \times F$ be a tubular neighborhood of $F$ in $\mathbb{R}^4$. 

\begin{definition}
A closed surface $S$ embedded in $N(F)$ is called a {\it branched covering surface-knot over $F$ of degree $N$} if it satisfies the followings. 

\begin{enumerate}[(1)]
\item
The restriction $p|_{S} \,:\, S \rightarrow F$ is a branched covering map of degree $N$, where $p\,:\, N(F) \to F$ is the natural projection. 

\item The number of points consisting $S \cap p^{-1}(x)$ is $N$ or $N-1$ for any point $x \in F$.
\end{enumerate}
Take a base point $x_0$ of $F$. 
Two branched covering surface-knots over $F$ of degree $N$ are {\it equivalent} if there is an ambient isotopy of $\mathbb{R}^4$ whose restriction to $N(F)=B^2 \times F$ is a fiber-preserving ambient isotopy rel $p^{-1}(x_0)$ taking one to the other. 

\end{definition}
\subsection{Chart presentation}\label{sec2-2}

Let $S$ be a branched covering surface-knot over a surface-knot $F$. A {\it surface diagram} of a surface-knot is the image of $F$ in $\mathbb{R}^3$ by a generic projection, equipped with the over/under information on sheets along each double point curve. In this paper, $F$ is an unknotted surface-knot in the standard form. We consider a surface diagram of $F$ as the image of $F$ by the projection $\mathbb{R}^4 \to \mathbb{R}^3$, $(x,y,z,t) \mapsto (x,y,z)$, which has no singularities. We identify the surface diagram of $F$ with $F$ itself, and we identify $N(F)$ with $I \times I \times F$, where the second $I$ is an interval in the fourth axis of $\mathbb{R}^4$. 

We explain a chart on a surface diagram $F$. 
Consider the singular set $\mathrm{Sing}(p_1(S))$ of the image of $S$ by the projection $p_1: I \times I \times F\to I \times F$, $(s,t,x) \mapsto (s,x)$. Perturbing $S$ if necessary, we assume that $\mathrm{Sing}(p_1(S))$ consists of double point curves, triple points, and branch points. Moreover, we assume that the singularity set of the image of $\mathrm{Sing}(p_1(S))$ by the projection to $F$ consists of a finite number of double points such that the preimages belong to double point curves of $\mathrm{Sing}(p_1(S))$. Thus, the image of $\mathrm{Sing}(p_1(S))$ by the projection to $F$ forms a finite graph $\Gamma$ on $F$ such that the degree of a vertex of $\Gamma$ is either $1$, $4$ or $6$. An edge of $\Gamma$ corresponds 
to a double point curve, and a vertex of degree $1$ (resp. $6$) 
corresponds to a branch point (resp. a triple point). 

For such a graph $\Gamma$ obtained from a branched covering surface-knot $S$, we assign orientations and labels to edges of $\Gamma$ as follows. We consider an oriented path $\rho$ in $F$ such that $\rho \cap \Gamma$ is a point $x$ of an edge $e$ of $\Gamma$. Then, $F \cap p^{-1} (\rho)$ is a classical $N$-braid with one crossing in $p^{-1}(\rho)$ such that $x$ corresponds to the crossing of the $N$-braid, where $N$ is the degree of $S$. Let $\sigma_{i}^{\epsilon}$ ($i \in \{1,2,\ldots, N-1\}$, 
$\epsilon \in \{+1, -1\}$) be the presentation of $S \cap p^{-1}(\rho)$. Then, assign $e$ the label $i$, and the orientation such that 
the normal vector of $\rho$ is coherent with (resp. opposite to) the orientation of $e$ if $\epsilon=+1$ (resp. $-1$), where the normal vector of $\rho$ is the vector $\vec{n}$ such that $(\vec{v}(\rho), \vec{n})$ is the orientation of $F$ for a tangent vector $\vec{v}(\rho)$ of $\rho$ at $x$. 
In this situation, we say that an oriented path $\rho$ {\it reads a braid} $\sigma_i^\epsilon$. 
This is the {\it chart of $S$}. 
\\
 
 For an unknotted surface-knot $F$ in the standard form, 
 we define a chart on a surface diagram of $F$ as follows, see \cite{Kamada02}. See \cite{N4} for a chart on a surface diagram which admits singularities. 
 
\begin{definition} 
Let $N$ be a positive integer. 
A finite graph $\Gamma$ on a surface diagram $F$ without singularities is called a {\it chart} of degree $N$ if 
it satisfies the following conditions. 

\begin{enumerate}[(1)]

 \item Every vertex has degree $1$, $4$, or $6$.
 
 \item Every edge of $\Gamma$ is oriented and labeled by an element of $\{1,2, \ldots, N-1\}$ such that the adjacent edges around each  of degree $1$, $4$, or $6$  
are oriented and labeled as shown in 
Fig. \ref{fig2}, 
where we depict a vertex of degree 1 by a black vertex, and a vertex of degree 
6 by a white vertex, and we call a vertex of degree $4$ a crossing.
 
 \end{enumerate}
\end{definition}

\begin{figure}[ht] 
\includegraphics*{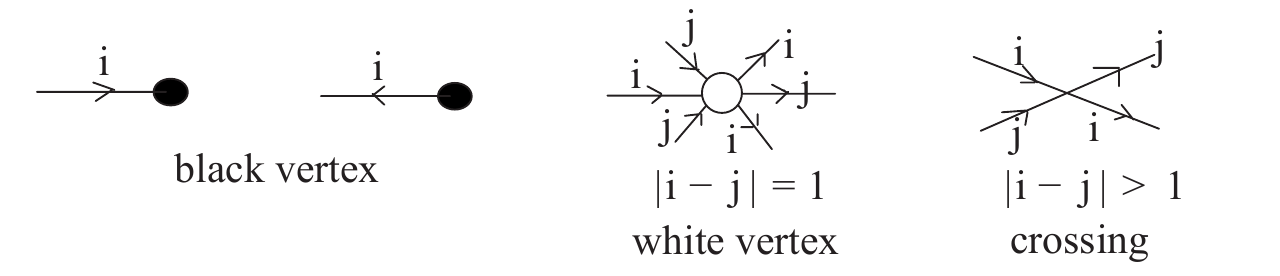}
\caption{Vertices in a chart, where $i \in \{1,\ldots,N-1\}$. There are three types. }
\label{fig2}
\end{figure}

A black vertex (resp. a white vertex) of a chart corresponds to a branch point (resp. a triple point) of the branched covering surface-knot presented by the chart.  We call an edge of a chart a {\it chart edge} or simply an {\it edge}. 
We regard diagonal edges connected with crossings as one edge 
with crossings, and we regard that a crossing is formed by a transverse intersection of two edges. 
A chart edge connected with no vertices or a chart edge with crossings 
is called a {\it chart loop} or simply a {\it loop}. A chart edge connected with two black vertices at endpoints is called a {\it free edge}. A chart is said to be {\it empty} if it is an empty graph. 

Among six edges connected with a white vertex, we call an edge which is the middle of three adjacent edges with the coherent orientation a {\it middle edge}, and we call an edge which is not a middle edge a {\it non-middle edge}; see Fig. \ref{fig4}. 
For edges connected with a white vertex, we call a pair of edges separated by two edges at each side {\it diagonal edges}. 

In this paper, for a white vertex $w$ such that the three adjacent edges oriented toward $w$ consist of two edges with the label $i$ and one edge with the label $j$, we call $w$ a white vertex of {\it type} $w_{i,j}$. For a crossing $c$ consisting of edges with the labels $i$ and $j$ such that the orientations of the edges form the right-handed orientation, we say $c$ a crossing of {\it type} $c_{i,j}$. See Fig. \ref{fig4}.  
 
\begin{figure}[ht] 
\includegraphics*[height=3cm]{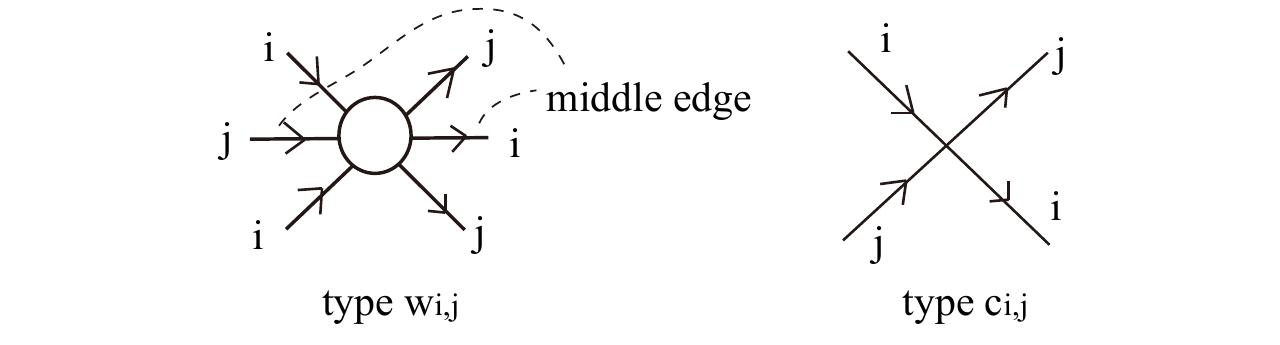}
\caption{A white vertex of type $w_{i,j}$ and a crossing of type $c_{i,j}$.}
\label{fig4}
\end{figure}

A branched covering surface-knot over a surface-knot $F$ is presented by a chart $\Gamma$ on a surface diagram of $F$ \cite{N4}. We present such a branched covering surface-knot by $(F, \Gamma)$. 
 
\subsection{C-moves}
  
{\it C-moves} are local moves of a chart, consisting of three types: CI-moves, CII-moves, and CIII-moves. 
Let $\Gamma$ and $\Gamma^{\prime}$ be two charts of the same degree on a surface diagram $F$. We say $\Gamma$ and $\Gamma'$ are related by a {\it CI-move}, {\it CII-move} or {\it CIII-move} if there exists a 2-disk $E$ in $F$ such that the loop $\partial E$ is in general position with respect to $\Gamma$ and $\Gamma^{\prime}$ and $\Gamma \cap (F-E)=\Gamma^{\prime} \cap (F-E)$, and the following conditions hold true.  
 
(CI) There are no black vertices in $\Gamma \cap E$ nor $\Gamma^{\prime} \cap E$. There are 7 types called CI-M1--CI-M4 moves, and CI-R1--CI-R3 moves as in Fig. \ref{fig5}. 

(CII) $\Gamma \cap E$ and $\Gamma' \cap E$ are as in Fig. \ref{fig5}, where $|i-j|>1$.

(CIII) $\Gamma \cap E$ and $\Gamma' \cap E$ are as in Fig. \ref{fig5}, where $|i-j|=1$, and the black vertex is connected with a non-middle edge of a white vertex. 
\\

\begin{figure}[ht]
\includegraphics*[width=13cm]{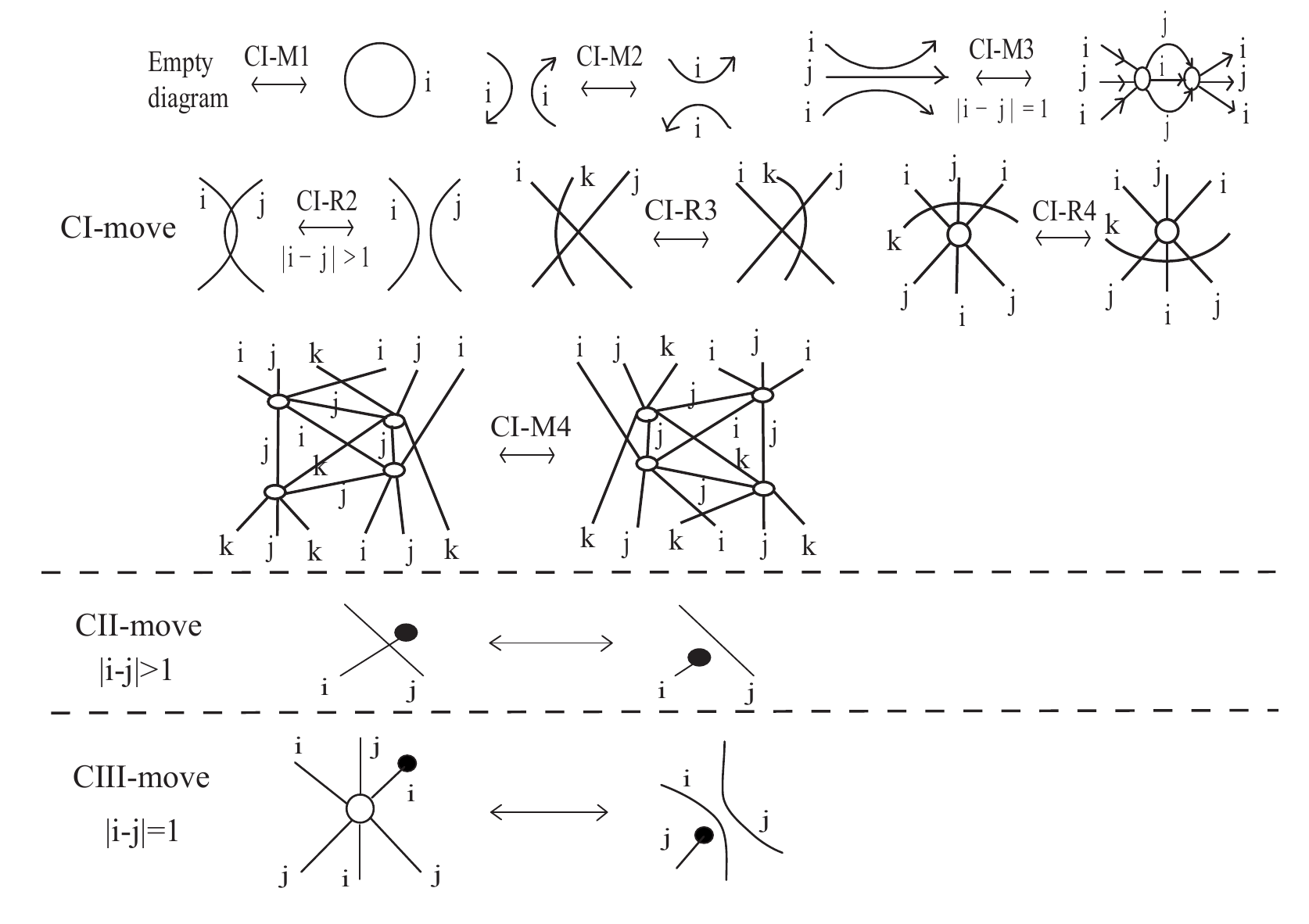}
\caption{C-moves. For simplicity, we omit orientations of some edges.} 
\label{fig5}
\end{figure}

Let $\Gamma$ and $\Gamma'$ be charts of the same degree on a surface diagram $F$. 
We say $\Gamma$ and $\Gamma'$ are {\it C-move equivalent} if they are related by a finite sequence of C-moves and ambient isotopies of $F$ rel $x_0$, where $x_0$ is a base point of $F$. 
For charts $\Gamma$ and $\Gamma'$,  
their presenting branched covering surface-knots are equivalent if the charts are C-move equivalent \cite{Kamada92, Kamada96, Kamada02}.

\subsection{Roseman moves}
 
{\it Roseman moves for surface diagrams with charts of the same degree} are defined by the original Roseman moves (see \cite{Roseman}) and moves for local surface diagrams with non-empty charts (see \cite{N4}), where we regard the diagrams for the original Roseman moves as equipped with empty charts. 
 
 For two surface diagrams with charts of the same degree, their presenting branched covering surface-knots are equivalent if they are related by a finite sequence of ambient isotopies of $\mathbb{R}^3$ and Roseman moves for surface diagrams with charts \cite{N4}.  
 
\section{Simplifying branched covering surface-knots}\label{sec3}

We denote by $h(a,b)$ for braids $a, b$ a 1-handle equipped with a chart without black vertices such that the cocore (resp. the orientation-reversed core loop) reads the presentation of $a$ (resp. $b$), and we assume that $h(a,b)$ is well-defined, that is, $a$ and $b$ commute. Unless otherwise said, we assume that 1-handles are attached to a fixed 2-disk $E$ in $F$ such that in $E$ there are no chart edges nor vertices except those of 1-handles. 
For 1-handles with charts $h(a_1, b_1)$, $\ldots$, $h(a_g, b_g)$, 
we denote by $\sum_{i=1}^g h(a_i,b_i)$ the result of the 1-handle addition. 
We say 1-handles with charts are {\it equivalent} if their presenting branched covering surface-knots are equivalent relative to a tubular neighborhood of $F-E$, and we use the notation \lq\lq $\sim$'' to denote equivalence relation. 
Let $N$ be the degree of branched covering surface-knots. 

\begin{lemma}\label{lem3-1}
For $N$-braids $a$ and $b$, we have 
\begin{equation}\label{eq3-1}
h(a,b) \sim h(a^{-1}, b^{-1}).
\end{equation}
In particular, 
\begin{equation}\label{eq3-2}
h(\sigma_i, e) \sim h(\sigma_i^{-1}, e), 
\end{equation}
where $i \in \{1, \ldots, N-1\}$. 
\end{lemma}

\begin{proof}
By rotating the 1-handle around and taking the orientation reversal of the original core loop as the new core loop, we have the result (see Fig. \ref{fig6}). See \cite[Lemma 4.1]{N5}. 
\end{proof}

\begin{figure}[ht]
\centering
\includegraphics*[height=2.5cm]{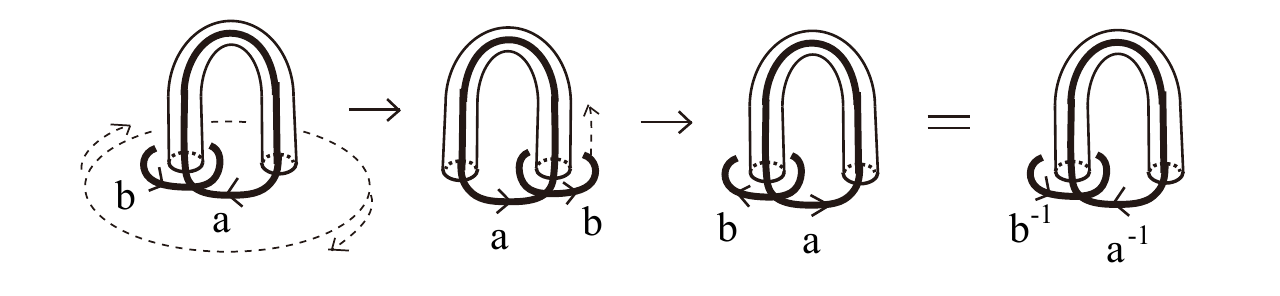}
\caption{The image of $h(a,b) \sim h(a^{-1},b^{-1})$, where actually the chart consists of chart edges, white vertices and crossings. }
\label{fig6}
\end{figure}

\begin{lemma}\label{lem3-4}
Together with $\sum_{i=1}^{N-1} h(\sigma_i, e)$, a 1-handle $h(e,e)$  transforms to $h(e, b) $ for any braid $b$:

\begin{equation}\label{eq3-7}
\sum_{i=1}^{N-1} h(\sigma_i,e) +h(e,e) \sim  \sum_{i=1}^{N-1} h(\sigma_i,e) +h(e,b).
\end{equation}
\end{lemma}

\begin{proof}
Let $c$ be a braid. 
By moving an end of $h(e,c)$ between the ends of $h(\sigma_i, e)$ and applying a CI-move as indicated in Fig. \ref{fig8}, $h(e,c)$ changes to $h(e, c\sigma_i^{\epsilon})$ ($\epsilon=-1$), and $h(\sigma_i, e)$ is unchanged. See also \cite[Lemma 4.6]{N5}. The sign $\epsilon$ changes to $+1$ by changing the orientation of the chart loop of $h(\sigma_i,e)$ by rotating it around to be $h(\sigma_i^{-1}, e)$. Hence, $h(e,c) +h(\sigma_i, e) \sim h(e, c \sigma_i^\epsilon)+h(\sigma_i,e)$, where $\epsilon \in \{+1, -1\}$. Since $\sum_{i=1}^{N-1} h(\sigma_i,e)$ covers all labels $1, \ldots, N-1$, by repeating this process several times to make chart loops surrounding the cocore of the first 1-handle such that the loops present the given braid $b$, we have the required result. 
\end{proof}

\begin{figure}[ht]
\centering
\includegraphics*[height=4.5 cm]{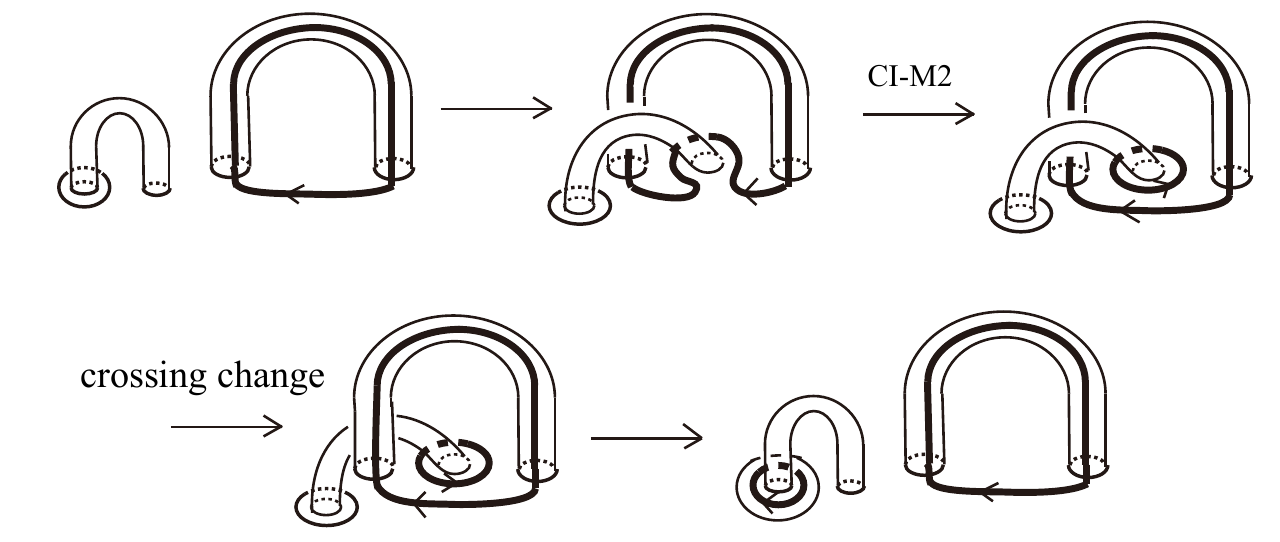}
\caption{$h(e,c) +h(\sigma_i, e) \sim h(e, c \sigma_i^{\epsilon})+h(\sigma_i,e)$. The orientations of chart loops are for the case $\epsilon=-1$. For simplicity, we omit the label of chart loops, and the part presenting $c$ is an image.}
\label{fig8}
\end{figure}

\begin{lemma}\label{lem3-5}
Together with $\sum_{i=1}^{N-1} h(\sigma_i, e)$, 
a 1-handle $h(e,e)$ transforms to a 1-handle with an empty chart such that each end is attached to anywhere. 
\end{lemma}

\begin{proof}
Let $E$ be the 2-disk where $\sum_{i=1}^{N-1} h(\sigma_i, e)$ and $h=h(e,e)$ are attached. Let $E_1$ and $E_2$ be 2-disks where we want to move the ends of $h$. Consider oriented paths from $E$ to $E_i$ ($i=1,2$) such that the intersection with the chart consists of the transverse intersections with chart edges. Let $b_1$ and $b_2$ be the braids presented by these paths, respectively. Then, by Lemma \ref{lem3-4}, by using $\sum_{i=1}^{N-1} h(\sigma_i, e)$, we transform $h$ to $h(e, b_1^{-1}b_2)$. We move the ends to $E_1$ and $E_2$ along the paths, and then $h(e, b_1^{-1}b_2)$ transforms to a 1-handle with chart loops surrounding the cocore presenting $b_1(b_1^{-1}b_2)b_2^{-1} \sim e$, which is equivalent to a 1-handle with an empty chart. In other words, by applying C1-M2 moves while moving the ends of $h$, we eliminate chart loops surrounding the cocore. See Fig. \ref{fig9}. 
\end{proof}

\begin{figure}[ht]
\centering
\includegraphics*[width=13cm]{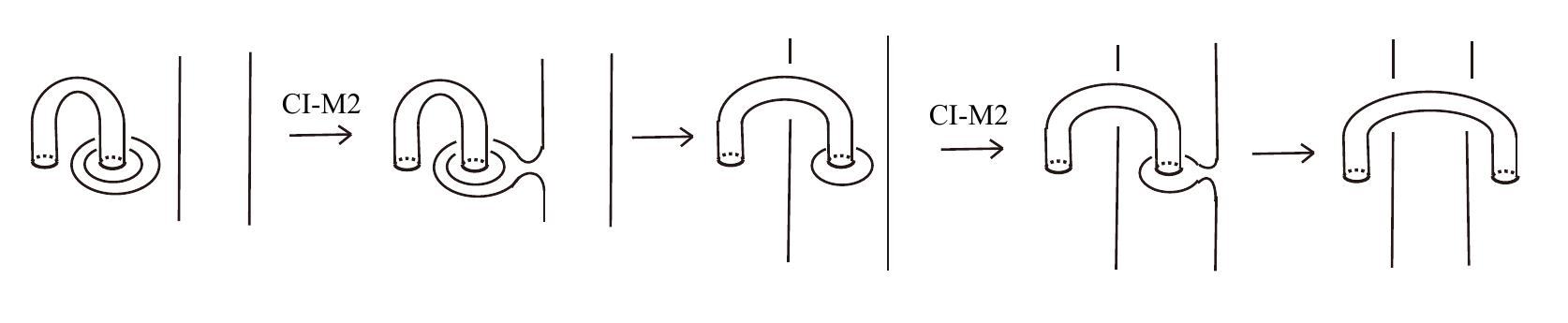}
\caption{Moving an end of a 1-handle to make a bridge. For simplicity, we omit labels and orientations of chart edges. }
\label{fig9}
\end{figure}

Let $E_1, E_2$ be 2-disks such that we have a 1-handle with an empty chart whose ends are attached to $E_1$ and $E_2$, respectively. We say that we {\it have a bridge} between $E_1$ and $E_2$.

\begin{lemma}\label{lem3-7}
Let $\rho$ be a chart loop with the label $i$ which has no
 crossings. Then, $\rho$ is eliminated by an addition of $h(\sigma_i, b)$ to a neighborhood of $\rho$, where $b$ is a braid, and the 1-handle is unchanged. 
In particular, we have 
\begin{equation}\label{eq3-9}
h(\sigma_i, b)+h(\sigma_i,e) \sim h(\sigma_i,b)+h(e,e)
\end{equation}
and 
\begin{equation}\label{eq3-10}
h(\sigma_i, b)+h(e, \sigma_i^\epsilon) \sim h(\sigma_i,b)+h(e,e), 
\end{equation}
where $i \in \{1, \ldots, N-1\}$, $\epsilon \in \{+1, -1\}$ and $b$ is a braid. 
\end{lemma}

\begin{proof}
By applying a CI-M2 move between the edges of the label $i$ of $\rho$ and $h(\sigma_i, b)$, and sliding an end of the 1-handle, we eliminate the chart loop $\rho$, and the 1-handle is unchanged. 
The relation (\ref{eq3-9}) is shown by taking $\rho$ as the chart loop along the core loop of the second 1-handle $h(\sigma_i,e)$ (see Fig. \ref{fig7}), and the relation (\ref{eq3-10}) is shown by taking $\rho$ as the chart loop surrounding the cocore of the second 1-handle $h(e, \sigma_i^\epsilon)$. 
\end{proof}

\begin{figure}[ht]
\centering
\includegraphics*[height=5cm]{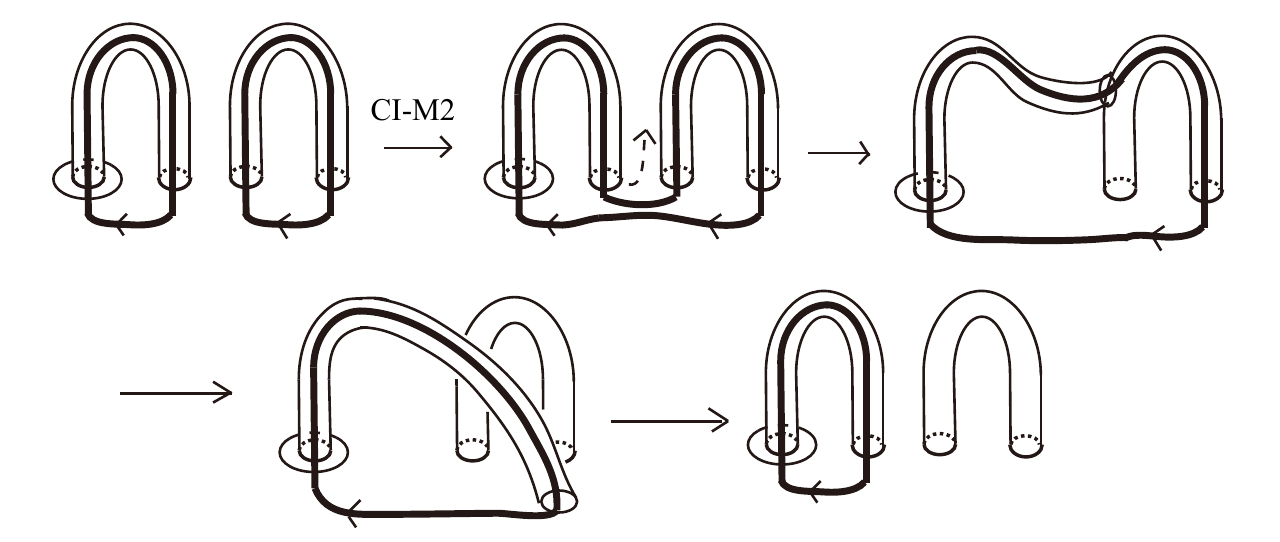}
\caption{$h(\sigma_i, b)+h(\sigma_i,e) \sim h(\sigma_i, b)+h(e, e)$. We omit the label of chart loops, and the part presenting $b$ is an image.}
\label{fig7}
\end{figure}

We remark that by repeating the inverse process of (\ref{eq3-10}) for $b=e$, $h(\sigma_i, e)+h(e,e) \sim h(\sigma_i, e)+h(e, \sigma_i^\epsilon)$ several times, we have Lemma \ref{lem3-4}. 

\begin{lemma}\label{lem3-8}
Let $\rho$ be a chart loop with the label $i$ with crossings such that $\rho$ as a closed path reads a braid $c$. Then, $\rho$ is eliminated by an addition of $h(\sigma_i, b)$ to a neighborhood of the initial point of $\rho$, where $b$ is a braid, and the 1-handle changes to $h(\sigma_i, bc^{-1})$. 
In particular, we have 
\begin{equation}\label{eq3-12}
h(\sigma_i, b)+h(\sigma_i, c) \sim h(\sigma_i,bc)+h(e,c), 
\end{equation}
and 
\begin{equation}\label{eq3-12a}
h(\sigma_i, b)+h(c, \sigma_i^\epsilon) \sim h(\sigma_i,bc^{-\epsilon})+h(c,e), 
\end{equation}
where $i \in \{1, \ldots, N-1\}$, $\epsilon \in \{+1, -1\}$ and $b, c$ are braids such that the 1-handles have only crossings for vertices. 
\end{lemma}

\begin{proof}
Similar to the proof of Lemma \ref{lem3-7}, by applying a CI-M2 move and sliding an end of the 1-handle, we eliminate the chart loop $\rho$. When the 1-handle slides across a crossing, the crossing is collected on the 1-handle. It follows that when the moving end comes back to a neighborhood of  the other end, the resulting 1-handle is $h(\sigma_i, bc^{-1})$. See Fig. \ref{fig11}. The relation (\ref{eq3-12}) is shown by taking $\rho$ as the chart loop around the core loop of the second 1-handle $h(\sigma_i,e)$. Since the braid $c$ of $h(\sigma_i, c)$ is read by the orientation-reversed core loop, the closed path $\rho$ reads $c^{-1}$ and the first 1-handle becomes $h(\sigma_i, bc)$. The relation (\ref{eq3-12a}) is shown by taking $\rho$ as the chart loop surrounding the cocore of the second 1-handle $h(e, \sigma_i^\epsilon)$. Since the braid $c$ of $h(c, \sigma_i^\epsilon)$ is read by the core loop, $\rho$ reads $c$ and the first 1-handle becomes $h(\sigma_i, bc^{-\epsilon})$. See also \cite[Figs. 18, 19 and Lemma 6.7]{N5}. 
\end{proof}

\begin{figure}[ht]
\centering
\includegraphics*[width=12cm]{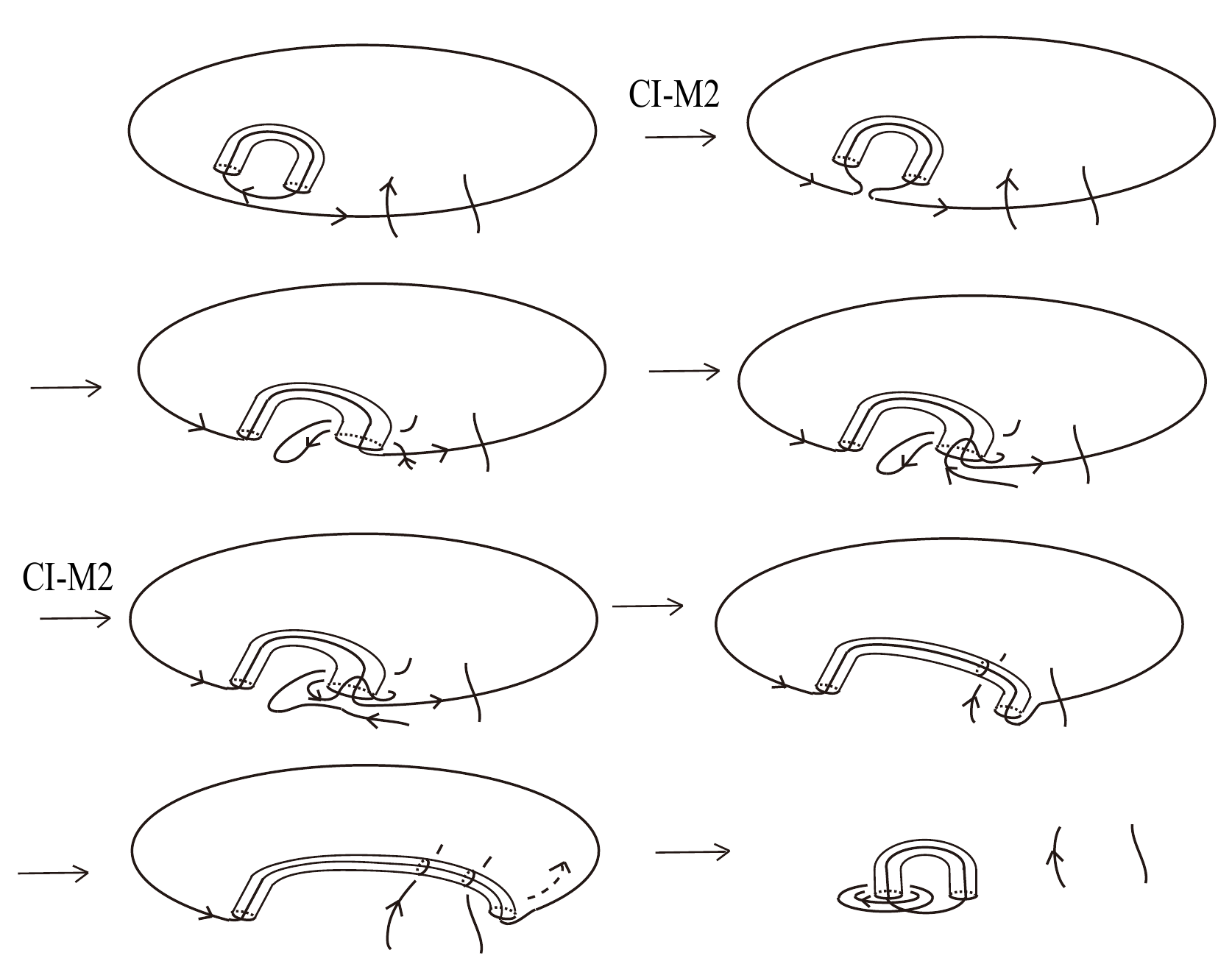}
\caption{Collecting crossings on a 1-handle. For simplicity, we omit labels of chart edges.}
\label{fig11}
\end{figure}

\begin{proof}[Proof of Theorem \ref{thm1-7}]
We show the second inequality (\ref{eq1-5}). 
If $b(\Gamma)=0$, then, by Lemma \ref{lem3-9}, white vertices consist of several pairs such that the types of each pair is $\{w_{i,j}, w_{j,i}\}$. By Lemma \ref{lem3-5}, by an addition of 1-handles $\sum_{i=1}^{N-1} h(\sigma_i,e)+h(e,e)$, we transform the 1-handle $h(e,e)$ to make a bridge between neighborhoods of a pair of white vertices of types $\{w_{i,j}, w_{j,i}\}$. Then, by an ambient isotopy and C1-moves, we eliminate the pair of white vertices as indicated in Fig. \ref{fig12}. Then, $\sum_{i=1}^{N-1} h(\sigma_i,e)$ is unchanged. 
By adding $h(e,e)$ and applying this process to every pair of types $\{w_{i,j}, w_{j,i}\}$, we deform $(F, \Gamma)$ to a branched covering surface-knot $(F', \Gamma')+\sum_{i=1}^{N-1} h(\sigma_i,e)$, where $F'=F+\sum_{j=1}^n h$ for a non-negative integer $n$ and $\Gamma'$ is a chart without white vertices. Since we obtained $(F', \Gamma')+\sum_{i=1}^{N-1} h(\sigma_i,e)$ from $(F, \Gamma)$ by an addition of $\sum_{i=1}^{N-1} h(\sigma_i,e)$ and $w(\Gamma)/2$ copies of $h(e,e)$, we see that $n=w(\Gamma)/2$. 

The chart $\Gamma'$ consists of several loops with $c(\Gamma)$ number of crossings. By Lemmas \ref{lem3-7} and \ref{lem3-8}, by using $\sum_{i=1}^{N-1} h(\sigma_i,e)$, we eliminate loops. 
The result is $(F, \emptyset)$ with $w(\Gamma)/2$ copies of $h(e,e)$ and $\sum_{i=1}^{N-1} h(\sigma_i,b_i)$ for some braids $b_i$ such that vertices are only crossings on $\sum_{i=1}^{N-1} h(\sigma_i, b_i)$. Then add several copies of $h(e,e)$ so that we have $c(\Gamma)$ copies of $h(e,e)$. Each crossing consists of a chart loop along the core loop and a chart loop $\rho$ parallel to the cocore of a 1-handle such that there is only one crossing on $\rho$. 
Let $E$ be the 2-disk where 1-handles are attached. Let $c$ be a crossing such that there exists a path from $E$ to $c$ which does not intersect with chart edges nor vertices. 
For such a crossing $c$, let $j$ be the label of the consisting loop $\rho$ parallel to the cocore. 
By the inverse process of Lemma \ref{lem3-7}, by using $h(\sigma_j, b_j)$, we change a 1-handle $h(e,e)$ to $h(\sigma_j, e)$. By an ambient isotopy, we move $h=h(\sigma_j,e)$ to a neighborhood of $c$, and by Lemma \ref{lem3-8},  we apply a CI-M2 move and slide an end of the 1-handle $h$ along $\rho$. Thus we collect the crossing $c$ on $h$, and then $h$ transforms to the form $h(\sigma_j, \sigma_k^\epsilon)$ ($|j-k|>1$, $\epsilon \in \{+1, -1\}$), attached on $E$. By repeating this process, we collect each crossing on a 1-handle, and thus we have 1-handles in the form $h(\sigma_j, \sigma_k^\epsilon)$ ($|j-k|>1$, $\epsilon \in \{+1, -1\}$), and $\sum_{i=1}^{N-1} h(\sigma_i, e)$, and $\max \{0, w(\Gamma)/2-c(\Gamma)\}$ copies of 1-handles $h(e,e)$, and we have a weak simplified form. 
Thus we have the second inequality $u_w(F, \Gamma) \leq \max\{w(\Gamma)/2, c(\Gamma)\}+N-1$. 

If $b(\Gamma)>0$ and $\Gamma$ has black vertices, then, by Lemma \ref{lem3-10}, the number of edges connected with black vertices with the orientation from a black vertex equals that with the orientation toward a black vertex. Since an edge with the label $i$ connected with a black vertex changes to an edge with the label $j$ ($|i-j|=1$) connected with a black vertex by increasing one white vertex as in Fig. \ref{fig13}, by increasing at most $(b(\Gamma)/2)(N-2)$ white vertices, we deform $\Gamma$ so that the number of edges connected with black vertices with the orientation from a black vertex and with the label $i$ equals that with the orientation toward a black vertex and with the label $i$. By considering the chart obtained from $\Gamma$ by making bridges by 1-handles between neighborhoods of black vertices, and eliminating black vertices and joining the edges formerly connected with them through 1-handles forming bridges, by Lemma \ref{lem3-9} we see that for this new chart and hence also $\Gamma$, the number of white vertices of type $w_{i,j}$ equals that of type $w_{j,i}$. Thus, by the same argument as in the case $b(\Gamma)=0$, we see that by an addition of $\sum_{i=1}^{N-1} h(\sigma_i,e)$ and $\lfloor \{w(\Gamma)+(b(\Gamma)/2)(N-2)\}/2\rfloor=\lfloor w(\Gamma)/2+(b(\Gamma)/4)(N-2)\rfloor$ copies of $h(e,e)$, we deform the union of $\Gamma$ and the charts on 1-handles to a chart without white vertices and moreover we have $\sum_{i=1}^{N-1} h(\sigma_i,e)$. The rest of the argument is the same as the case $b(\Gamma)=0$: we eliminate chart loops by using $\sum_{i=1}^{N-1} h(\sigma_i,e)$, and by an addition of 1-handles $h(e,e)$ we move each crossing to be on one 1-handle. By applying CII-moves if necessary, black vertices become the endpoints of free edges, and the resulting chart consists of several free edges and 1-handles with chart loops; thus we have the weak simplified form and we have the first inequality  (\ref{eq1-4}): $u_w(F, \Gamma) \leq \max\{\lfloor w(\Gamma)/2+(b(\Gamma)/4)(N-2)\rfloor, c(\Gamma)\}+N-1$. 
\end{proof}

\begin{figure}[ht]
\centering
\includegraphics*[width=13cm]{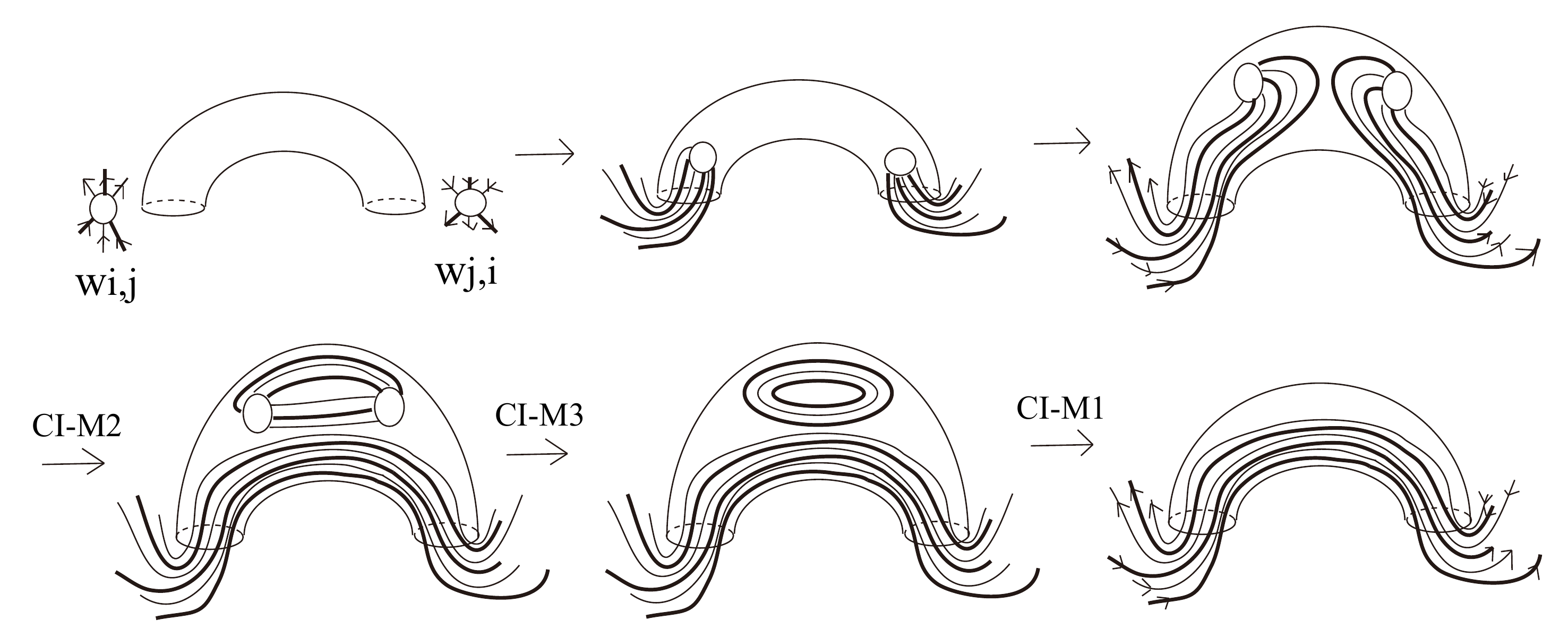}
\caption{Elimination of a pair of  white vertices $w_{i,j}$ and $w_{j,i}$. We omit labels and orientations of chart edges.}
\label{fig12}
\end{figure}

\begin{figure}[ht]
\centering
\includegraphics*[height=3cm]{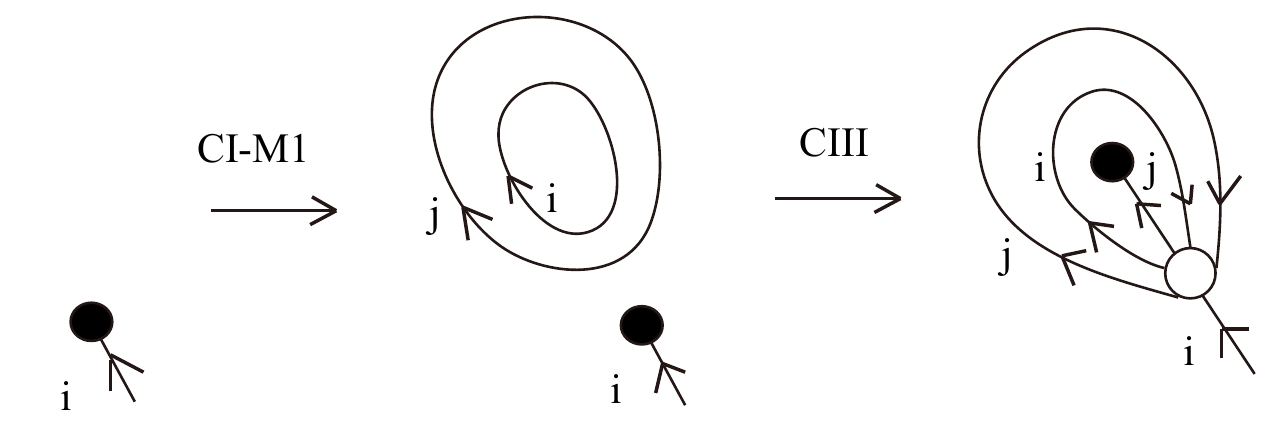}
\caption{Changing the label of the chart edge connected with a black vertex. The orientations of chart edges are an example. }
\label{fig13}
\end{figure}

\begin{lemma}\label{lem3-9}
Let $\Gamma$ be a chart of degree $N$ such that $b(\Gamma)=0$, where $b(\Gamma)$ is the number of black vertices. Then, the number of white vertices of type $w_{i,j}$ equals that of type $w_{j,i}$, where $|i-j|=1$, $i,j \in \{1,\ldots, N-1\}$.  
\end{lemma} 

\begin{proof}
In the set of edges connected with white vertices, 
an edge with the label $1$ appears only as an edge connected with a white vertex of type $w_{1,2}$ or $w_{2,1}$. Around a white vertex $w$ of type $w_{1,2}$ (resp. $w_{2,1}$), the number of edges with the label $1$ with the orientation toward $w$ is $2$ (resp. $1$), and the number of those with the orientation from $w$ is $1$ (resp. $2$). Since there are no black vertices, the number of edges around white vertices with the label $1$ and the orientation toward the vertex equals that with the orientation from the vertex, where we count an edge twice if each of its endpoint is connected with a white vertex. It follows that the number of white vertices of type $w_{1,2}$ equals that of type $w_{2,1}$. 
Thus, repeating the same argument for edges of label $i$ for $i=2,3,\ldots, N-2$ in this order, we obtain the required result. 
\end{proof}

\begin{lemma}\label{lem3-10}
For a chart, the number of edges connected with black vertices with the orientation from a black vertex equals that with the orientation toward a black vertex. 
\end{lemma}
\begin{proof}
We see the result from the fact that around a crossing/white vertex, the number of the edges with the orientation from the crossing/white vertex equals that with the orientation toward it. 
\end{proof}

\section{To more simplified forms I}\label{sec4}

In this section, we show Theorem \ref{thm1-8}. Let $N$ be the degree of branched covering surface-knots. 

\begin{lemma}\label{lem4-3a}
We have 
\[
h(\sigma_j, b_j)+h(\sigma_k, b_k)+h(\sigma_j, \sigma_k^{\epsilon}) \sim h(\sigma_j, b_j)+h(\sigma_k, b_k)+h(\sigma_k, \sigma_j^{-\epsilon}),
\]
where $|j-k|>1$ $(j,k \in \{1, \ldots, N-1\})$, $\epsilon\in\{+1, -1\}$, and $b_j$ and $b_k$ are braids. 
\end{lemma}

\begin{proof}
Assume that we have $h(\sigma_j, b_j)+h(\sigma_k, b_k)+h(\sigma_j, \sigma_k^{\epsilon})$. 
Put $h_1=h(\sigma_j, b_j)$, $h_2=h(\sigma_k, b_k)$, and $h_3=h(\sigma_j, \sigma_k^{\epsilon})$. 
By Lemma \ref{lem3-8}, by applying a CI-M2 move to the edges with the label $j$ of $h_1$ and $h_3$, and sliding an end of $h_1$ along $h_3$, $h_1$ becomes $h(\sigma_j, b_j \sigma_k^\epsilon)$ and $h_3$ becomes $h(e, \sigma_k^\epsilon)$, and we have 
\[
h(\sigma_j, b_j\sigma_k^\epsilon)+h(\sigma_k, b_k)+h(e,  \sigma_k^{\epsilon}).
\]
By applying Lemma \ref{lem3-8} to $h_3=h(e,  \sigma_k^{\epsilon})$ and $h_2=h(\sigma_k, b_k)$, $h_3$ becomes $h(e, e)$, and $h_2$ is unchanged. We have 
\begin{equation}\label{eq4-1a}
h(\sigma_j, b_j\sigma_k^\epsilon)+h(\sigma_k, b_k)+h(e,e).
\end{equation}
By applying the inverse process of Lemma \ref{lem3-8} to $h_2=h(\sigma_k, b_k)$ and $h_3=h(e,e)$, $h_3$ becomes $h(\sigma_k, e)$ and $h_2$ is unchanged, and we have 
\[
h(\sigma_j, b_j\sigma_k^\epsilon)+h(\sigma_k, b_k)+h(\sigma_k,e).
\]
By applying Lemma \ref{lem3-8} to $h_1=h(\sigma_j, b_j\sigma_k^\epsilon)$ and $h_3=h(\sigma_k,e)$, by applying a CI-M2 move  to the edges with the label $k$ and collecting a crossing on $h_3$, $h_3$ becomes $h(\sigma_k, \sigma_j^{-\epsilon})$ and $h_1$ becomes $h(\sigma_j, b_j)$, and we have 
\[
h(\sigma_j, b_j)+h(\sigma_k, b_k)+h(\sigma_k,\sigma_j^{-\epsilon}),
\]
which is the required result.
\end{proof}

\begin{lemma}\label{lem4-5}
For $j<k$ with $|j-k|>1$, 
\begin{equation}\label{eq4-01}
\sum_{i=1}^{N-1} h(\sigma_i, b_i)+h(\sigma_j, \sigma_k^\epsilon) \sim \sum_{i=1}^{N-1} h(\sigma_i, b_i)+h(\sigma_1, \sigma_3^\epsilon), 
\end{equation}
and for $j>k$ with $|j-k|>1$, 
\begin{equation}\label{eq4-02}
\sum_{i=1}^{N-1} h(\sigma_i, b_i)+h(\sigma_j, \sigma_k^\epsilon) \sim \sum_{i=1}^{N-1} h(\sigma_i, b_i)+h(\sigma_1, \sigma_3^{-\epsilon}), 
\end{equation}
where $j,k \in \{1, \ldots, N-1\}$, $\epsilon \in \{+1, -1\}$,  and $b_i$ are braids. 
\end{lemma}

\begin{proof}
It suffices to show that for $j, k$ with $|j-k|>1$ and $|j-k+1|>1$, 
\begin{equation}\label{eq4-1}
\sum_{i=1}^{N-1} h(\sigma_i, b_i)+h(\sigma_j, \sigma_k^\epsilon) \sim \sum_{i=1}^{N-1} h(\sigma_i, b_i)+h(\sigma_j, \sigma_{k-1}^\epsilon)
\end{equation}
 and for $j, k$ with $|j-k|>1$ and $|j-k-1|>1$,
\begin{equation}\label{eq4-2}
\sum_{i=1}^{N-1} h(\sigma_i, b_i)+h(\sigma_j, \sigma_k^\epsilon) \sim \sum_{i=1}^{N-1} h(\sigma_i, b_i)+h(\sigma_{j-1}, \sigma_k^\epsilon). 
\end{equation}
The relation (\ref{eq4-01}) follows from (\ref{eq4-1}) and (\ref{eq4-2}). 
Since $ h(\sigma_3, \sigma_1^{\epsilon}) \sim h(\sigma_1, \sigma_3^{-\epsilon})$ by Lemma \ref{lem4-3a}, (\ref{eq4-02}) follows from (\ref{eq4-1}) and (\ref{eq4-2}). 

We show (\ref{eq4-1}). 
Since $h(\sigma_{k-1}, b_{k-1}) \sim h(\sigma_{k-1},  \sigma_{j}^{-\epsilon} \sigma_{j}^{\epsilon}b_{k-1})$ and $h(\sigma_k, b_k) \sim h(\sigma_k, \sigma_{j}^{-\epsilon} \sigma_{j}^{\epsilon}b_k)$, by Lemma \ref{lem4-2} or Lemma \ref{lem3-8}, by sliding $h(\sigma_j, \sigma_k^\epsilon)$ along the cocores and collecting the crossings, $\sum_{i=1}^{N-1} h(\sigma_i, b_i)+h(\sigma_j, \sigma_k^\epsilon)$ deforms to 
\[
\sum_{i=1}^{N-1} h(\sigma_i, b_i)+h(\sigma_j, \sigma_k^\epsilon \sigma_{k-1}^\epsilon \sigma_k^\epsilon \sigma_{k-1}^{-\epsilon} \sigma_k^{-\epsilon}). 
\]
Since $\sigma_k^\epsilon \sigma_{k-1}^\epsilon \sigma_k^\epsilon \sigma_{k-1}^{-\epsilon} \sigma_k^{-\epsilon} \sim (\sigma_k^\epsilon \sigma_{k-1}^\epsilon \sigma_k^\epsilon \sigma_{k-1}^{-\epsilon} \sigma_k^{-\epsilon} \sigma_{k-1}^{-\epsilon}) \sigma_{k-1}^{\epsilon} \sim e \sigma_{k-1}^\epsilon \sim \sigma_{k-1}^\epsilon$, 
we have 
\[
\sum_{i=1}^{N-1} h(\sigma_i, b_i)+h(\sigma_j, \sigma_{k-1}^\epsilon).  
\]
Thus, we have (\ref{eq4-1}). The second relation (\ref{eq4-2}) is obtained by using Lemma \ref{lem4-3a} and (\ref{eq4-1}) as follows. By Lemma \ref{lem4-3a}, $\sum_{i=1}^{N-1} h(\sigma_i, b_i)+h(\sigma_j, \sigma_k^\epsilon) \sim \sum_{i=1}^{N-1} h(\sigma_i, b_i)+h(\sigma_k, \sigma_j^{-\epsilon})$. By (\ref{eq4-1}), this is equivalent to $\sum_{i=1}^{N-1} h(\sigma_i, b_i)+h(\sigma_k, \sigma_{j-1}^{-\epsilon})$. 
Since $h(\sigma_k, \sigma_{j-1}^{-\epsilon})\sim h(\sigma_{j-1}, \sigma_k^{\epsilon})$, we have (\ref{eq4-2}).  
\end{proof}

\begin{proof}[Proof of Theorem \ref{thm1-8}]
By Lemma \ref{lem3-8}, $h(\sigma_1, \sigma_3^{-1})+h(\sigma_1, \sigma_3)$ is equivalent to $h(\sigma_1, \sigma_3^{-1} \sigma_3)+h(e, \sigma_3)$, which is equivalent to $h(\sigma_1, e)+h(e, \sigma_3)$. These deform to $h(e,e)+h(e,e)$ by an addition of $\sum_{i=1}^{N-1} h(\sigma_i, e)$. Hence, by this argument and 
Theorem \ref{thm1-2} and Lemma \ref{lem4-5}, by an addition of finitely many 1-handles in the form $h(\sigma_i, e)$, or $h(e,e)$ $(i \in \{1, \ldots,N-1\})$, to appropriate places in $F$, $(F, \Gamma)$ deforms to  
\begin{equation*} 
(F, \Gamma_0) +\sum_{i=1}^{N-1}h(\sigma_i, e)+\sum h(\sigma_1, \sigma_3^\epsilon) +\sum h(e,e),
\end{equation*}
where $\epsilon$ is either $+1$ or $-1$. 
\end{proof}

\section{To more simplified forms II}\label{sec5}

In this section, we show Theorem \ref{thm1-9} and Corollaries \ref{thm1-11}, \ref{thm1-10}. Let $N$ be the degree of branched covering surface-knots. 

\begin{lemma}\label{lem4-1}
Together with $\sum_{i=1}^{N-1} h(\sigma_i, b_i)$, a 1-handle $h(a,b)$ moves anywhere, where $b_i, a, b$ are braids. 
\end{lemma}

\begin{proof}
By moving the whole 1-handle instead of an end of the 1-handle as in the proofs of Lemmas \ref{lem3-4} and \ref{lem3-5}, we have the required result. See Fig. \ref{fig10}. 
\end{proof}

\begin{figure}[ht]
\centering
\includegraphics*[width=13cm]{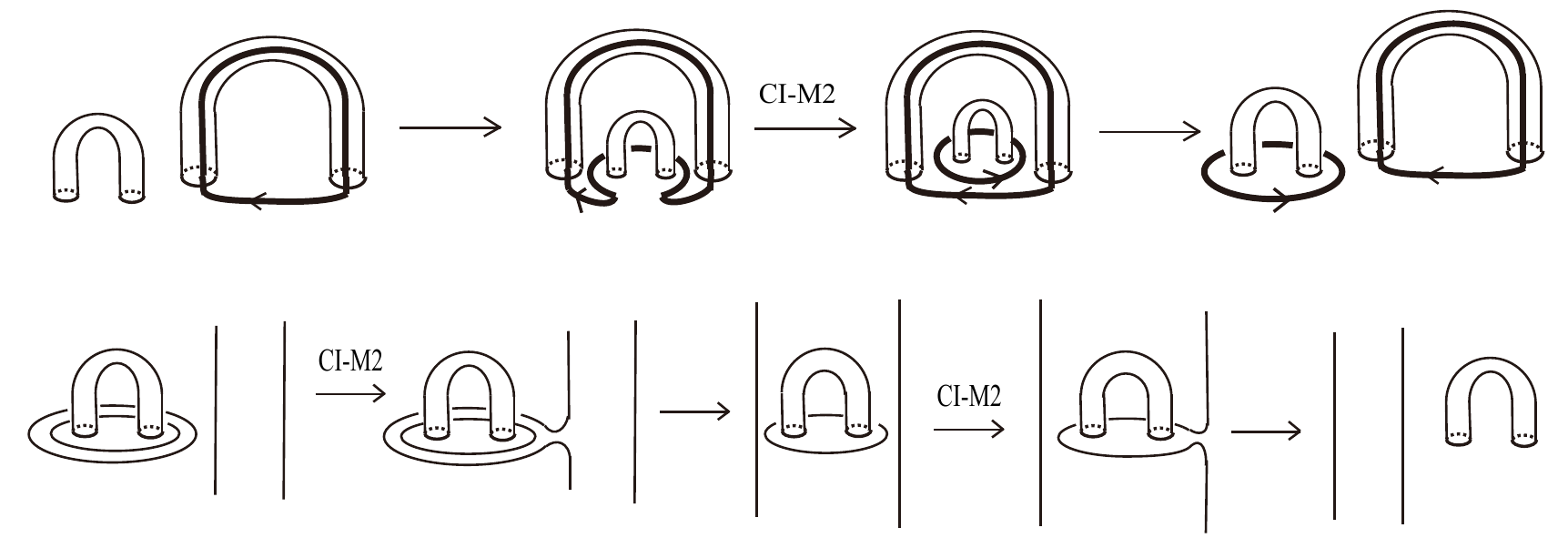}
\caption{Moving a 1-handle to make a chart loop surrounding it (upper figures), and moving a 1-handle through chart edges (lower figures). The orientations of chart loops are an example. For simplicity, we omit labels and orientations of some chart edges. }
\label{fig10}
\end{figure}
 
\begin{lemma}\label{lem4-2}
We have 
\begin{equation*}
\sum_{i \neq j} h(\sigma_i, b_i)+h(\sigma_j, b \sigma_k^\epsilon b')+h(\sigma_k, c) \sim \sum_{i \neq j} h(\sigma_i, b_i)+h(\sigma_j, b b')+h(\sigma_k, c\sigma_j^{-\epsilon}), 
\end{equation*}
where $|j-k|>1$, $\epsilon \in \{+1, -1\}$, and $b_i, b, b', c$ are braids such that the 1-handles have only crossings as vertices. 
\end{lemma}

\begin{proof}
Assume that we have $\sum_{i \neq j} h(\sigma_i, b_i)+h(\sigma_j, b \sigma_k^\epsilon b')+h(\sigma_k, c)$. 
By Lemma \ref{lem4-1}, we move $h=h(\sigma_k,c)$ to a neighborhood $E$ of the chart loop $\rho$ with the label $k$ parallel to the cocore of $h'=h(\sigma_j, b \sigma_k^\epsilon b')$ presenting $\sigma_k$ of $b \sigma_k^\epsilon b'$. Then, by Lemma \ref{lem3-8}, by applying a CI-M2 move and sliding an end of $h$ along $\rho$, we eliminate $\rho$ to collect on $h$ the crossing formed by $\rho$ and the chart loop along the core loop of $h'$ with the label $j$. The 1-handle $h$ becomes $h(\sigma_k, c\sigma_j^{-\epsilon})$ attached to $E$, and $h'$ becomes $h(\sigma_j, bb')$. By CI-M2 moves, when $h$ comes back to the original place, it is surrounded by chart loops presenting $b$. We eliminate these loops by Lemma \ref{lem3-7}, by using the other 1-handles $\sum_{i \neq j} h(\sigma_i, b_i)+h(\sigma_j, bb')$. Then, we have $\sum_{i \neq j} h(\sigma_i, b_i)+h(\sigma_j, b b')+h(\sigma_k, c\sigma_j^{-\epsilon})$, which is the required result. 
\end{proof}

\begin{lemma}\label{lem4-3}
We have 
\[
\sum_{i\neq j} h(\sigma_i, b_i)+h(\sigma_j, b \sigma_k^\epsilon b')+h(\sigma_j, c) \sim \sum_{i \neq j} h(\sigma_i, b_i)+h(\sigma_j, b b')+h(\sigma_j, c\sigma_k^{\epsilon}), 
\]
where $|j-k|>1$, $b_i, b, b', c$ are braids such that the 1-handles have only crossings, and $\epsilon \in \{+1, -1\}$. 
\end{lemma}

\begin{proof}
Assume that we have $\sum_{i\neq j} h(\sigma_i, b_i)+h(\sigma_j, b \sigma_k^\epsilon b')+h(\sigma_j, c)$. 
By Lemma \ref{lem3-8}, by applying a CI-M2 move to the edges with the label $j$ of $h=h(\sigma_j, b \sigma_k^\epsilon b')$ and $h'=h(\sigma_j, c)$, and sliding an end of $h$ along $h'$, $h$ becomes $h(\sigma_j, b \sigma_k^\epsilon b'c)$ and   $h'$ becomes $h(e, c)$. We have 
\[
\sum_{i\neq j} h(\sigma_i, b_i)+h(\sigma_j, b \sigma_k^\epsilon b'c)+h(e,c).
\]
By using $\sum_{i\neq j} h(\sigma_i, b_i)+h(\sigma_j, b \sigma_k^\epsilon b'c)$, by Lemma \ref{lem3-8}, we eliminate all chart loops parallel to the cocore of $h(e,c)$, and $h(e,c)$ becomes $h(e,e)$. We have 
\[
\sum_{i\neq j} h(\sigma_i, b_i)+h(\sigma_j, b \sigma_k^\epsilon b'c)+h(e,e).
\]
By applying the inverse process of Lemma \ref{lem3-7} to $h(\sigma_k, b_k)$ and $h(e,e)$, $h(e,e)$ becomes $h(\sigma_k, e)$. We have 
\[
\sum_{i\neq j} h(\sigma_i, b_i)+h(\sigma_j, b \sigma_k^\epsilon b'c)+h(\sigma_k,e).
\]
By Lemma \ref{lem4-1}, we move $h(\sigma_k,e)$ to a neighborhood of the chart loop $\rho$ with the label $k$ parallel to the cocore of $h'=h(\sigma_j, b \sigma_k^\epsilon b'c)$ presenting $\sigma_k^\epsilon$ of $b \sigma_k^\epsilon b'c$. By Lemma \ref{lem3-8}, we eliminate $\rho$ to collect on $h(\sigma_k, e)$ the crossing formed by $\rho$ and the chart loop along the core loop of $h'$ with the label $j$. We have 
\[
\sum_{i\neq j} h(\sigma_i, b_i)+h(\sigma_j, b b'c)+h(\sigma_k, \sigma_j^{-\epsilon}).
\]
By Lemma \ref{lem3-8}, by applying a CI-M2 move and sliding an end of $h(\sigma_j, b b'c)$ around the cocore of $h(\sigma_k, \sigma_j^{-\epsilon})$, we have 
\[
\sum_{i\neq j} h(\sigma_i, b_i)+h(\sigma_j, bb'c\sigma_k^{\epsilon})+h(\sigma_k, e). 
\]
By applying Lemma \ref{lem3-7} to $h(\sigma_k, b_k)$ and $h(\sigma_k,e)$, we have 
\[
\sum_{i\neq j} h(\sigma_i, b_i)+h(\sigma_j, bb'c\sigma_k^{\epsilon})+h(e, e). 
\]
By the inverse processes of Lemma \ref{lem3-7} to $\sum_{i\neq j} h(\sigma_i, b_i)+h(\sigma_j, bb'c\sigma_k^{\epsilon})$ and 
$h(e,e)$, $h(e,e)$ becomes $h(e, c\sigma_k^{\epsilon})$.  
We have 
\[
\sum_{i\neq j} h(\sigma_i, b_i)+h(\sigma_j, bb'c\sigma_k^{\epsilon})+h(e, c\sigma_k^{\epsilon}). 
\]
By the inverse process of Lemma \ref{lem3-8} to $h=h(\sigma_j, bb'c\sigma_k^{\epsilon})$ and $h'=h(e,c\sigma_k^{\epsilon})$, 
$h$ becomes $h(\sigma_j, bb')$ and $h'$ becomes $h(\sigma_j, c\sigma_k^{\epsilon})$, and we have 
\[
\sum_{i\neq j} h(\sigma_i, b_i)+h(\sigma_j, bb')+h(\sigma_j, c\sigma_k^{\epsilon}), 
\]
which is the required result.
\end{proof}

For a crossing $c$ of a chart $\Gamma$, 
let $s(c)=+1$ (resp. $-1$) if $c$ is of type $c_{i,j}$ for $i<j$ (resp. $i>j$), where $|i-j|>1$, $i,j \in \{1, \ldots, N-1\}$. 
Let $s(\Gamma)$ be the sum of $s(c)$ for crossings of $\Gamma$. 

\begin{proposition}\label{prop1-9}
Let $(F, \Gamma)$ be a branched covering surface-knot of degree $N$ such that $w(\Gamma)=0$, where $w(\Gamma)$ is the number of white vertices. 
Put $s=s(\Gamma)$. Assume that $F=F_0+h$ for a surface-knot $F_0$ and a 1-handle $h$. Then, by an addition of $\sum_{i=1}^{N-1} h(\sigma_i, e)$,  
$(F, \Gamma)$ deforms to 
\begin{equation}\label{eq-prop}
(F_0, \Gamma_0)+h(\sigma_1, \sigma_3^{-s})+\sum_{i=2}^{N-1} h(\sigma_i, e)+ \sum_{j=1}^n h(e,e), 
\end{equation}
where $\Gamma_0$ is a chart consisting of several free edges, and $n \geq 1$. Further, by an addition of $h(\sigma_1, \sigma_3^{-1})+\sum_{i=2}^{N-1} h(\sigma_i, e)$,  
$(F, \Gamma)$ deforms to 
\begin{equation}\label{eq-prop2}
(F_0, \Gamma_0)+h(\sigma_1, \sigma_3^{-s-1})+\sum_{i=2}^{N-1} h(\sigma_i, e)+ \sum_{j=1}^n h(e,e),  
\end{equation}
where $\Gamma_0$ is a chart consisting of several free edges, and $n \geq 1$. 
\end{proposition}

\begin{remark}
In this paper, since $F$ is an unknotted surface-knot, $F=F_0+h$ if and only if $g(F)>0$, where $g(F)$ is the genus of $F$. We write the condition this way so that we can extend the statement to a general case for a knotted surface $F$.  
\end{remark}

\begin{proof}[Proof of Proposition \ref{prop1-9}]
We show (\ref{eq-prop}). We add $\sum_{i=1}^{N-1} h(\sigma_i, e)$ to $(F, \Gamma)$. Since $w(\Gamma)=0$, $\Gamma$ 
consists of chart loops with crossings and several free edges. By applying Lemmas \ref{lem3-7} and \ref{lem3-8}, by using $\sum_{i=1}^{N-1} h(\sigma_i, e)$, we eliminate chart loops.  
Then we have 
$\sum_{i=1}^{N-1} h(\sigma_i, b_i)$, where $b_i$ are braids such that there are only crossings for vertices on 1-handles, and more than zero copies of $h(e,e)$. 
From now on, we consider 
\[
\sum_{i=1}^{N-1} h(\sigma_i, b_i)+h(e,e). 
\]
By applying the inverse process of Lemma \ref{lem3-8} to $h(\sigma_{N-1}, b_{N-1})$ and $h(e,e)$, we have 
\[
\sum_{i=1}^{N-1} h(\sigma_i, b_i)+h(\sigma_{N-1},e). 
\]
By Lemmas \ref{lem4-2} and \ref{lem4-3}, we collect all crossings such that one of the two consisting edges has the label $N-1$: we collect crossings of type $c_{i, N-1}$ and $c_{N-1, i}$ ($i=1, \ldots, N-3$) on the last 1-handle  $h=h(\sigma_{N-1}, e)$, by the following method. We use a similar argument as in the proof of Lemma \ref{lem4-5} where we showed (\ref{eq4-1}). 
By Lemmas \ref{lem4-2} and \ref{lem4-3}, the order of collecting crossings can be chosen arbitrarily. 
First we collect on $h$ crossings of type $c_{1, N-1}$ and $c_{N-1, 1}$. Then we collect on $h$ crossings of type $c_{2, N-1}$ and $c_{N-1, 2}$. When we collect a crossing of type $c_{2, N-1}$ or $c_{N-1, 2}$, and $h$ becomes $h(\sigma_{N-1},  \sigma_1^m \sigma_2^\epsilon)$ for an integer $m$ and $\epsilon \in \{+1, -1\}$, then we deform the type of the crossing to $c_{1, N-1}$ or $c_{N-1, 1}$ as follows. We denote by $\sum_{i=1}^{N-1} h(\sigma_i, b_i')$ the 1-handles after we collected the crossing, where $b_i'$ $(i=1, \ldots, N-1)$ are braids: we have 
\[
\sum_{i=1}^{N-1} h(\sigma_i, b_i')+h(\sigma_{N-1}, \sigma_1^m \sigma_2^\epsilon),  
\]
where $m$ is an integer and $\epsilon \in \{+1, -1\}$. 
Since $h(\sigma_1, b_1') \sim h(\sigma_1,  \sigma_{N-1}^{-\epsilon} \sigma_{N-1}^{\epsilon}b_{1}')$ and $h(\sigma_2, b_2') \sim h(\sigma_2, \sigma_{N-1}^{-\epsilon} \sigma_{N-1}^{\epsilon}b_2')$, by Lemma \ref{lem4-2} or Lemma \ref{lem3-8}, by sliding $h$ along the cocores and collecting 4 crossings of type $\{ c_{N-1,1}, c_{N-1,2}, c_{1, N-1}, c_{2, N-1}\}$, $\sum_{i=1}^{N-1} h(\sigma_i, b_i')+h(\sigma_{N-1}, \sigma_1^m \sigma_2^\epsilon)$ deforms to 
\[
\sum_{i=1}^{N-1} h(\sigma_i, b_i')+h(\sigma_{N-1}, \sigma_1^m \sigma_2^\epsilon \sigma_{1}^\epsilon \sigma_2^\epsilon \sigma_1^{-\epsilon} \sigma_2^{-\epsilon}). 
\]
Since $\sigma_2^\epsilon \sigma_{1}^\epsilon \sigma_2^\epsilon \sigma_{1}^{-\epsilon} \sigma_2^{-\epsilon} \sim (\sigma_2^\epsilon \sigma_{1}^\epsilon \sigma_2^\epsilon \sigma_{1}^{-\epsilon} \sigma_2^{-\epsilon} \sigma_{1}^{-\epsilon}) \sigma_{1}^{\epsilon} \sim e \sigma_{1}^\epsilon \sim \sigma_{1}^\epsilon$, 
we have 
\[
\sum_{i=1}^{N-1} h(\sigma_i, b_i')+h(\sigma_{N-1}, \sigma_1^{m+\epsilon}). 
\]
Applying this process each time we collect a crossing of type $c_{2, N-1}$ or $c_{N-1, 2}$, we collect crossings of type $c_{2, N-1}$ and $c_{N-1, 2}$ on $h$ and moreover deform the types of the collected crossings to $c_{1, N-1}$ or $c_{N-1, 1}$. Then we collect crossings of type $c_{3, N-1}$ and $c_{N-1, 3}$ on $h$. When we collect such a crossing, and $h$ becomes $h(\sigma_{N-1},  \sigma_1^m \sigma_3^\epsilon)$ for an integer $m$ and $\epsilon \in \{+1, -1\}$, then we deform the type of the crossing to $c_{1, N-1}$ and $c_{N-1, 1}$ as follows. For simplicity, we use the same notation $\sum_{i=1}^{N-1} h(\sigma_i, b_i')$ to denote the 1-handles after we collected the crossing: we have 
\[
\sum_{i=1}^{N-1} h(\sigma_i, b_i')+h(\sigma_{N-1}, \sigma_1^m \sigma_3^\epsilon),  
\]
where $m$ is an integer and $\epsilon \in \{+1, -1\}$. 
Since $h(\sigma_2, b_2') \sim h(\sigma_2,  \sigma_{N-1}^{-\epsilon} \sigma_{N-1}^{\epsilon}b_{2}')$ and $h(\sigma_3, b_3') \sim h(\sigma_3,  \sigma_{N-1}^{-\epsilon} \sigma_{N-1}^{\epsilon}b_{3}')$, by Lemma \ref{lem4-2} or Lemma \ref{lem3-8}, by sliding $h$ along the cocores and collecting 4 crossings, $\sum_{i=1}^{N-1} h(\sigma_i, b_i')+h(\sigma_{N-1}, \sigma_1^m \sigma_3^\epsilon)$ deforms to 
\[
\sum_{i=1}^{N-1} h(\sigma_i, b_i')+h(\sigma_{N-1}, \sigma_1^m \sigma_3^\epsilon \sigma_{2}^\epsilon \sigma_3^\epsilon \sigma_2^{-\epsilon} \sigma_3^{-\epsilon}),  
\]
and hence
\[
\sum_{i=1}^{N-1} h(\sigma_i, b_i')+h(\sigma_{N-1}, \sigma_1^m \sigma_{2}^\epsilon). 
\]
Then, by the above argument, we deform this to 
 \[
\sum_{i=1}^{N-1} h(\sigma_i, b_i')+h(\sigma_{N-1}, \sigma_1^{m+\epsilon}). 
\]
By repeating similar processes and collecting crossings of type $c_{i, N-1}$ and $c_{N-1, i}$ ($i=1, \ldots, N-3$) and deforming the types to $c_{1, N-1}$ or $c_{N-1, 1}$ one by one, 
we have 
\[
\sum_{i=1}^{N-2} h(\sigma_i, b_i')+h(\sigma_{N-1},e)+h(\sigma_{N-1}, \sigma_1^m),  
\]
where $b_i'$ ($i=1, \ldots, N-2$) are braids which do not contain a chart loop with the label $N-1$, and $m$ is an integer. 
By Lemma \ref{lem3-7}, by using $h(\sigma_{N-1}, \sigma_1^m)$,  $h(\sigma_{N-1}, e)$ becomes $h(e,e)$, and we have 
\[
\sum_{i=1}^{N-2} h(\sigma_i, b_i')+h(\sigma_{N-1}, \sigma_1^m)+h(e,e).
\]
By the inverse process of Lemma \ref{lem3-7}, we have 
\[
\sum_{i=1}^{N-2} h(\sigma_i, b_i')+h(\sigma_{N-1}, \sigma_1^m)+h(\sigma_{N-2},e). 
\]
Since $\sum_{i=1}^{N-2} h(\sigma_i, b_i')+h(\sigma_{N-2},e)$ consists of charts with the labels in $\{1, \ldots, N-2\}$, we collect crossings of type $c_{i, N-2}$ and $c_{N-2, i}$ $(i=1, \ldots, N-4)$ on the last 1-handle $h(\sigma_{N-2},e)$ and moreover we deform the types of the collected crossings to $c_{1, N-2}$ or $c_{N-2, 1}$, and we have 
\[
\sum_{i=1}^{N-3} h(\sigma_i, b_i'')+h(\sigma_{N-2}, \sigma_1^n)+h(\sigma_{N-1}, \sigma_1^m)+h(e,e),  
\]
where $b_i''$ ($i=1, \ldots, N-3$) are braids which do not contain a chart loop with the label $N-2$ nor $N-1$, and $m, n$ are integers. 
By repeating this process, we have 
\[
h(\sigma_1, e)+h(\sigma_2, e)+
\sum_{i=3}^{N-1} h(\sigma_i, \sigma_1^{m_i})+h(e,e),  
\]
where $m_i$ ($i=3, \ldots, N-1$) are integers. 
By the inverse process of Lemma \ref{lem3-7}, to $h(\sigma_1, e)$ and $h(e,e)$, we have 
\[
h(\sigma_1, e)+h(\sigma_2, e)+
\sum_{i=3}^{N-1} h(\sigma_i, \sigma_1^{m_i})+h(\sigma_1,e).  
\]
By Lemma \ref{lem3-8}, we collect all crossings on the last 1-handle $h(\sigma_1, e)$. Since $h(\sigma_i, \sigma_1^{m_i}) \sim h(\sigma_i, \sigma_1^{-\epsilon} \sigma_1^\epsilon \sigma_1^{m_i})$ $(i=3, \ldots, N-1, \epsilon \in \{+1, -1\})$, by collecting 4 crossings of type $\{ c_{1, i-1}, c_{1, i}, c_{i-1,1}, c_{i, 1}\}$ $(i=4, \ldots, N-1)$ if necessary when we collect each crossing, we deform the types of the collected crossings to $c_{1,3}$ or $c_{3,1}$. 
Then, we have 
\[
\sum_{i =1}^{N-1} h(\sigma_i, e)+ h(\sigma_1, \sigma_3^{s'}), 
\]
where $s'$ is an integer. 
By the inverse process of Lemma \ref{lem3-7} to $h(\sigma_1, \sigma_3^{s'})$ and $h(\sigma_1,e)$, we deform $h(\sigma_1, e)$ to $h(e,e)$, and we have 
\begin{equation}\label{eq-0208}
h(\sigma_1, \sigma_3^{s'})+\sum_{i =2}^{N-1} h(\sigma_i, e)+ h( e,e), 
\end{equation}
where $s'$ is an integer. 
When we deform the type of a crossing to $c_{1,3}$ or $c_{3,1}$, we use only the relations $\sigma_i\sigma_i^{-1} \sim \sigma_i^{-1} \sigma_i \sim e$ and $\sigma_{i}\sigma_{j}\sigma_i \sim \sigma_{j} \sigma_i \sigma_{j}$ ($|i-j|=1$). Thus, for the types of crossings we have the following relations: $\{c_{i,j}, c_{j,i}\}$ changes to no crossing ($|i-j|>1$), and $\{c_{i,k}, c_{j,k}, c_{i,k}\}$ changes to $\{c_{j.k}, c_{i,k}, c_{j,k}\}$ ($|i-j|=1$, $|i-k|>1$, $|j-k|>1$), and $\{c_{k,i}, c_{k,j}, c_{k,i}\}$ changes to $\{c_{k,j}, c_{k,i}, c_{k,j}\}$ ($|i-j|=1$, $|i-k|>1$, $|j-k|>1$), and vice versa. The number $s(\Gamma')$ is invariant under these relations, where $\Gamma'$ is the chart of the result of the 1-handle addition to $(F, \Gamma)$. Thus $s(\Gamma')$ does not change after the application of Lemmas \ref{lem3-7}, \ref{lem3-8}, \ref{lem4-2} or \ref{lem4-3}, or the deformation we applied in this argument. Since the crossing on $h(\sigma_1, \sigma_3^{\epsilon})$ is of type $c_{3,1}$ (resp. $c_{1,3}$) if $\epsilon=+1$ (resp. $-1$), $s(h(\sigma_1, \sigma_3^{\epsilon}))=-\epsilon$, where we denote by the same notation $h(\sigma_1, \sigma_3^{\epsilon})$ the chart of $h(\sigma_1, \sigma_3^{\epsilon})$. Hence 
we see that $s(\Gamma')$ for the chart $\Gamma'$ of (\ref{eq-0208}) is $-s'$, which equals $s(\Gamma)$; this implies that $s'=-s(\Gamma)$, and we have (\ref{eq-prop}). 

We show (\ref{eq-prop2}). When we add $h(\sigma_1, \sigma_3^{-1})+\sum_{i=2}^{N-1} h(\sigma_i, e)$ to $(F, \Gamma)$, let $\Gamma'$ denote the chart of the result of the 1-handle addition. Then, the number $s(\Gamma')$ increases by 1 from $s(\Gamma)$, and we have $-s'=s(\Gamma)+1$ for $s'$ in (\ref{eq-0208}); thus we have  (\ref{eq-prop2}).  
\end{proof}

The number $s(\Gamma)=\sum_c s(c)$, where $c$ runs over crossings, is not invariant under CI-moves  (see Fig. \ref{fig5}). 
The number $s(\Gamma)$ is invariant under CI-R2, CI-R3, and CI-R4 moves, but it is not invariant under CI-M4 moves. 

\begin{remark}\label{rem-4}
In \cite[Proposition 1.10]{N5}, we announced that $u(F, \Gamma) \leq u_w(F, \Gamma)+c_{\mathrm{alg}}(\Gamma)$, where $c_{\mathrm{alg}}(\Gamma)$ is the sum of the absolute values of the number of crossings of type $c_{i,j}$ minus that of type $c_{j,i}$ ($i<j$, $i,j \in \{1, \ldots, N-1\}$):  $c_{\mathrm{alg}}(\Gamma)=\sum_{i<j} |\sum_{\text{$c$ is of type $c_{i,j}$ or $c_{j,i}$}} s(c)|$, where $c$ runs over crossings, and $s(c)=+1$ (resp. $-1$) if $c$ is of type $c_{i,j}$  for $i>j$ (resp. for $i<j$). 
In the proof given in \cite{N5}, we 
showed it under the conditions that the number of crossings are less than or equal to that of $\Gamma$ after we deformed $(F, \Gamma)$ to a weak simplified form $(F', \Gamma')$, and further that there is an inclusion from the set of crossings of $\Gamma'$ with type information to that of $\Gamma$. However, even if the number of white vertices $w(\Gamma)$ is zero, $c_{\mathrm{alg}}(\Gamma)$ is not invariant under CI-M4 moves. Hence we correct this inequality to Conjecture \ref{conj2} (\ref{conj5}). 
\end{remark}

By using a CI-M4 move, we have the following lemma.  

\begin{lemma}\label{lemma1-9}
We have 
\begin{multline}\label{eq4-6a}
h(\sigma_i, b_i)+h(\sigma_j, b_j)+h(\sigma_k, b_k)+h(e,e) \\
\sim h(\sigma_i, b_i \sigma_k^{-1})+h(\sigma_j, b_j)+h(\sigma_k, b_k \sigma_i) +h(e,e), 
\end{multline}
where $|i-j|=|j-k|=1$ and $|i-k|>1$ $(i, j, k \in \{1, \ldots, N-1\})$, and $b_i, b_j, b_k$ are braids.  
In particular, we have 
\begin{multline}\label{eq4-7a}
h(\sigma_1, \sigma_3^n)+h(\sigma_2, e)+h(\sigma_3, e) +h(e,e)\\
 \sim h(\sigma_1, \sigma_3^{n-2})+h(\sigma_2, e)+h(\sigma_3, e) +h(e,e),   
\end{multline}
where $n$ is an integer. 
\end{lemma}

\begin{proof}
We show (\ref{eq4-6a}). 
By the inverse processes of Lemma \ref{lem3-7}, $h(\sigma_i, b_i)+h(\sigma_j, b_j)+h(\sigma_k, b_k) +h(e,e)$ deforms to 
\[
h(\sigma_i, b_i)+h(\sigma_j, b_j)+h(\sigma_k, b_k) +h(\sigma_i \sigma_j \sigma_i \sigma_k \sigma_j \sigma_i,e). 
\]
\begin{figure}[ht]
\includegraphics*[width=12cm]{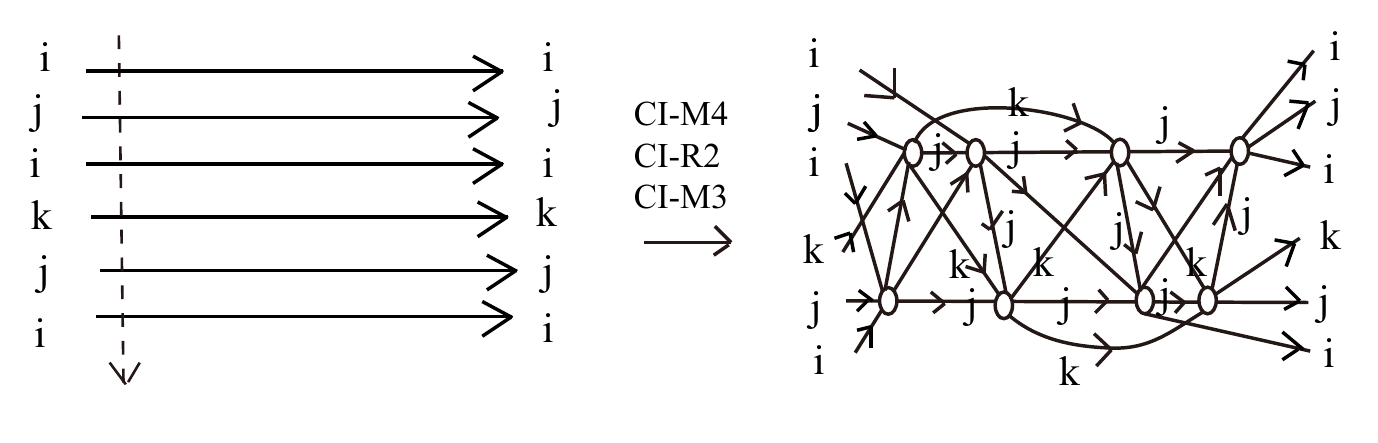}
\caption{The chart presenting the braid $\sigma_i \sigma_j \sigma_i \sigma_k \sigma_j \sigma_i$ ($|i-j|=|j-k|=1$, $|i-k|>1$) is equivalent to a chart with 8 white vertices and 6 crossings. We omit orientations of some edges. } 
\label{fig14a}
\end{figure}
\begin{figure}[ht]
\includegraphics*[height=8cm]{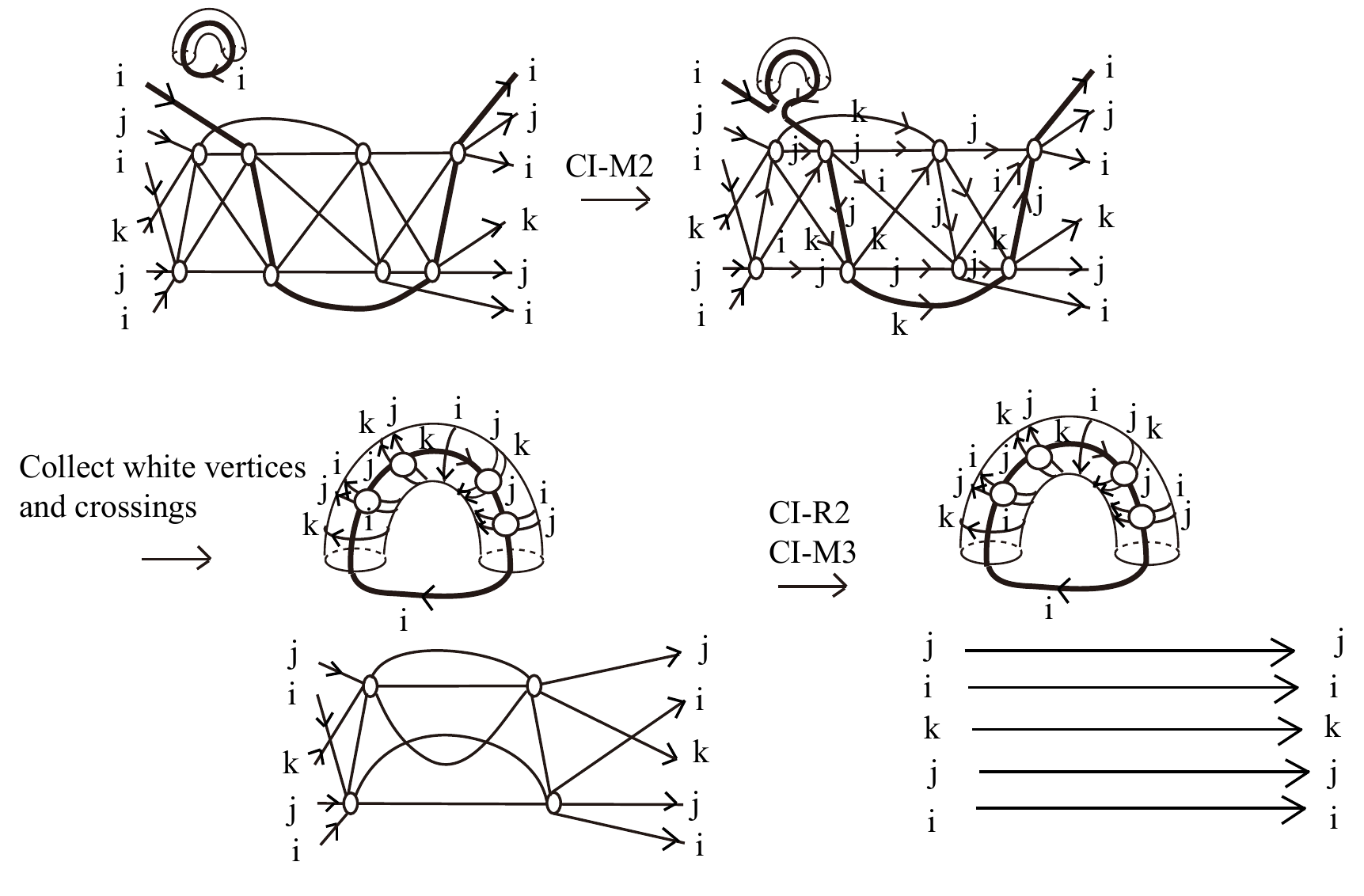}
\caption{Collecting white vertices and crossings on a 1-handle, where the braid is closed and on another 1-handle, and $|i-j|=|j-k|=1$ and $|i-k|>1$.} 
\label{fig15a}
\end{figure}
By CI-M4, CI-R2 and CI-M3 moves, $h=h(\sigma_i \sigma_j \sigma_i \sigma_k \sigma_j \sigma_i,e)$ deforms to have 8 white vertices and 6 crossings as indicated in Fig. \ref{fig14a}. We assume that $h(\sigma_i, b_i)+h(\sigma_j, b_j)+h(\sigma_k, b_k)$ are attached to a neighborhood of the edge of $h$ as in the first figure of Fig. \ref{fig15a}. 
Since the edge is a non-middle edge with the label $i$ of a white vertex $w$, we apply a CI-M2 move between the edge and $h'=h(\sigma_i, b_i)$, and we slide an end of $h'$ to collect $w$ on $h'$; see Fig. \ref{fig2a}. We slide $h'$ along the diagonal edges indicated by bold edges in Fig. \ref{fig15a}. Since the diagonal edges consist of non-middle edges, $h'$ collects 4 white vertices and 2 crossings (see Fig. \ref{fig15a}), and we have 
\begin{multline*}
h(\sigma_i, b_i \sigma_k^{-1} (\sigma_j^{-1} \sigma_i^{-1}) (\sigma_k^{-1}\sigma_j^{-1}) \sigma_i (\sigma_j \sigma_k) (\sigma_i \sigma_j) )+h(\sigma_j, b_j)+h(\sigma_k, b_k)\\ +h(\sigma_j \sigma_i \sigma_k \sigma_j \sigma_i,e). 
\end{multline*}
By Lemma \ref{lem3-7}, we eliminate chart loops of the last 1-handle, and we have 
\[
h(\sigma_i, b_i \sigma_k^{-1} (\sigma_j^{-1} \sigma_i^{-1}) (\sigma_k^{-1}\sigma_j^{-1}) \sigma_i (\sigma_j \sigma_k) (\sigma_i \sigma_j) )+h(\sigma_j, b_j)+h(\sigma_k, b_k) +h(e, e). 
\]
By the inverse process of Lemma \ref{lem3-7}, $h(e,e)$ becomes $h=h(\sigma_i, e)$, and by a similar argument as in the proof of Lemma \ref{lem4-2}, we collect on $h$ a crossing presenting the first $\sigma_i$ of $b_i \sigma_k^{-1} (\sigma_j^{-1} \sigma_i^{-1}) (\sigma_k^{-1}\sigma_j^{-1}) \sigma_i (\sigma_j \sigma_k) (\sigma_i \sigma_j)$ of the first 1-handle. Since the crossing consists of a chart edge with the label $k$ along the core and a chart loop with the label $i$ parallel to the cocore, we have 
\[
h(\sigma_i, b_i \sigma_k^{-1} (\sigma_j^{-1} \sigma_i^{-1}) (\sigma_k^{-1}\sigma_j^{-1})  (\sigma_j \sigma_k) (\sigma_i \sigma_j) )+h(\sigma_j, b_j)+h(\sigma_k, b_k) +h(\sigma_i, \sigma_k^{-1}). 
\]
The first 1-handle becomes $h(\sigma_i, b_i \sigma_k^{-1})$, and we have 
\[
h(\sigma_i, b_i \sigma_k^{-1})+h(\sigma_j, b_j)+h(\sigma_k, b_k) +h(\sigma_i, \sigma_k^{-1}). 
\]
By Lemma \ref{lem4-2} or Lemma \ref{lem3-8}, $h(\sigma_k, b_k) +h(\sigma_i, \sigma_k^{-1})$ deforms to $h(\sigma_k, b_k \sigma_i) +h(\sigma_i, e)$, and we have 
\[
h(\sigma_i, b_i \sigma_k^{-1})+h(\sigma_j, b_j)+h(\sigma_k, b_k \sigma_i) +h(\sigma_i, e). 
\]
By the inverse process of Lemma \ref{lem3-7}, $h(\sigma_i, e)$ becomes $h(e,e)$, and we have 
(\ref{eq4-6a}). 

We show (\ref{eq4-7a}). By (\ref{eq4-6a}), we have 
\begin{multline*}
h(\sigma_1, \sigma_3^n)+h(\sigma_2, e)+h(\sigma_3, e) +h(e,e)\\
 \sim h(\sigma_1, \sigma_3^{n-1})+h(\sigma_2, e)+h(\sigma_3, \sigma_1) +h(e,e). 
\end{multline*}
By applying Lemma \ref{lem4-2} or Lemma \ref{lem3-8} to $h(\sigma_1, \sigma_3^{n-1})$ and $h(\sigma_3, \sigma_1)$, we have (\ref{eq4-7a}). 
\end{proof}

\begin{proof}[Proof of Theorem \ref{thm1-9}]
By an addition of $u_w(F, \Gamma)$ 1-handles, we deform $(F, \Gamma)$ to a weak simplified form $(F', \Gamma')$. 
Then, by Proposition \ref{prop1-9} (\ref{eq-prop}) and Lemma \ref{lemma1-9} (\ref{eq4-7a}), by an addition of $\sum_{i=1}^{N-1}h(\sigma_i, e)$ if $u_w(F, \Gamma)>0$, and $\sum_{i=1}^{N-1}h(\sigma_i, e)+h(e,e)$ if $u_w(F, \Gamma)=0$, respectively, we deform $(F', \Gamma')$ to the required form (\ref{eq1-8x}) or (\ref{eq1-9x}). 

If $w(\Gamma)=0$, then, by Proposition \ref{prop1-9} (\ref{eq-prop}) and Lemma \ref{lemma1-9} (\ref{eq4-7a}), 
$(F, \Gamma)$ deforms to (\ref{eq1-8x}) if $c(\Gamma)$ is even, and (\ref{eq1-9x}) if $c(\Gamma)$ is odd.  
If $b(\Gamma)=0$, then, by the argument in the proof of Theorem \ref{thm1-7}, by an addition of 1-handles and eliminating white vertices by pairs $\{ w_{i,j}, w_{j,i} \}$, $(F, \Gamma)$ deforms to a chart without white vertices such that the types of the crossings are unchanged. 
Hence, by Proposition \ref{prop1-9} (\ref{eq-prop}) and Lemma \ref{lemma1-9} (\ref{eq4-7a}), 
we have (\ref{eq1-8y}) and (\ref{eq1-9y}). 
\end{proof}

Remark that in the proof of Theorem \ref{thm1-9}, we used $\max\{1, u_w(F, \Gamma)\}+N-1$ 1-handles for the general case, but 
we used $\max \{1, \lfloor w(\Gamma)/2 \rfloor \}+N-1$ 1-handles for the case when $w(\Gamma)=0$ or $b(\Gamma)=0$. 

\begin{proof}[Proof of Corollary \ref{thm1-11}]
We show (\ref{eq1-12}). Let $\Gamma$ be a chart such that $w(\Gamma)=0$. 
Then, $u_w(F, \Gamma) \leq 1+(N-1)=N$ follows from Proposition \ref{prop1-9} (\ref{eq-prop}) and Lemma \ref{lemma1-9} (\ref{eq4-7a}). 
The case of $u(F, \Gamma)$ is shown as follows. Let $c(\Gamma)$ be the number of crossings of $\Gamma$. 
If $c(\Gamma)$ is even, then, by an addition of $\sum_{i=1}^{N-1}h(\sigma_i, e)+h(e,e)$, $(F, \Gamma)$ deforms to a simplified form (\ref{eq1-8y}), and hence $u(F, \Gamma) \leq N$. If $c(\Gamma)$ is odd, then, by Proposition \ref{prop1-9} (\ref{eq-prop2}) and Lemma \ref{lemma1-9} (\ref{eq4-7a}), by an addition of $h(\sigma_1, \sigma_3^{-1})+\sum_{i=2}^{N-1}h(\sigma_i, e)+h(e,e)$, $(F, \Gamma)$ deforms to a simplified form. Thus $u(F, \Gamma) \leq N$. 

We show (\ref{eq1-10}). 
By the proof of Theorem \ref{thm1-7}, by an addition of at most $\lfloor w(\Gamma)/2+b(\Gamma)(N-2)/4 \rfloor$ 1-handles, $\Gamma$ deforms to have no white vertices. Then, by the same argument as the case $w(\Gamma)=0$, 
we have $u_w(F, \Gamma), u(F, \Gamma) \leq \max\{ 1, \lfloor w(\Gamma)/2+b(\Gamma)(N-2)/4\rfloor \}+N-1$. 
\end{proof}

\begin{proof}[Proof of Corollary \ref{thm1-10}]
By the same argument as in the proof of Corollary \ref{thm1-11}, by Proposition \ref{prop1-9} and Lemma \ref{lemma1-9} (\ref{eq4-7a}), by an addition of $\max\{1, u_w(F, \Gamma)\}+N-1$ 1-handles, we have a simplified form, and hence $u(F, \Gamma) \leq \max\{1, u_w(F, \Gamma)\}+N-1$.
\end{proof}

\section*{Acknowledgements}
The author would like to thank Professors Seiichi Kamada and Shin Satoh for their helpful comments. The author was partially supported by JSPS KAKENHI Grant Numbers 15H05740 and 15K17532.

\end{document}